\documentclass[12pt,a4paper,twoside,reqno]{amsart}

\usepackage{amsmath,amssymb,amsfonts,amsthm,mathrsfs}
\usepackage{enumitem}
\usepackage{times,hyperref,color}
\usepackage{graphicx}
\usepackage{cite}
\usepackage[toc,page]{appendix}
\usepackage{bm}

\newcommand{\C}{\mathbb{C}}
\newcommand{\N}{\mathbb{N}}\newcommand{\R}{\mathbb{R}}

\newcommand{\D}{\mathcal{D}}
\newcommand{\DZ}{\mathcal{D}_Z}
\newcommand{\DoZ}{\mathcal{D}_{\oZ}}
\newcommand{\DU}{\mathcal{D}_U}
\newcommand{\G}{\mathcal{G}}
\newcommand{\bz}{\bm{z}}
\newcommand{\Jn}{{\rm J}^n}
\newcommand{\n}{{}^{(n)}}

\newcommand{\B}{\mathcal{B}}
\newcommand{\bZ}{\bm{Z}}
\newcommand{\jet}{\mathrm{j}}
\newcommand{\hB}{\widehat{\B}}
\newcommand{\hrho}{\widehat{\rho}}

\newcommand{\hD}{\widehat{D}}
\newcommand{\bH}{\textbf{H}}
\newcommand{\mK}{\mathcal{K}}
\newcommand{\bp}{{\bm{p}}}

\newcommand{\omu}{\overline{\mu}}
\newcommand{\bmu}{\bm{\mu}}
\newcommand{\mcu}{\nu}
\newcommand{\mcv}{\psi}

\newcommand{\sctypez}{\delta}
\newcommand{\sctypeoz}{\kappa}
\newcommand{\sctypeu}{\lambda}

\newcommand{\oxi}{\overline{\xi}}
\newcommand{\pp}[2]{\frac{\partial #1}{\partial #2}}
\newcommand{\vv}{\mathbf{v}}
\newcommand{\CMJ}{{\frak J}}
\newcommand{\CMoJ}{\overline \CMJ}
\newcommand{\CMK}{{\frak K}}
\newcommand{\CML}{{\frak L}}
\newcommand{\CMR}{{\frak R}}
\let\mathcal\mathscr

\newtheorem{Theorem}{Theorem}[section]

\newtheorem{Proposition}[Theorem]{Proposition}

\newtheorem{Lemma}[Theorem]{Lemma}

\theoremstyle{definition}
\newtheorem{Definition}[Theorem]{Definition}

\newtheorem{Remark}[Theorem]{Remark}

\newcommand{\oz}{\bar{z}}
\newcommand{\oZ}{\overline{Z}}
\newcommand{\oF}{\overline{F}}

\def\abs#1{|\,{#1}\,|}

\def\hexnumber#1{\ifcase#1 0\or1\or2\or3\or4\or5\or6\or7\or8\or9\or
 A\or B\or C\or D\or E\or F\fi}

\edef\msbhx{\hexnumber\symAMSb}   

\mathchardef\emptyset="0\msbhx3F
\def\i{\,{\rm i}\,}
\def\ostrut#1#2{\hbox{\vrule height #1pt depth #2pt width 0pt}}

\makeatletter
\@namedef{subjclassname@2020}{%
  \textup{2020} Mathematics Subject Classification}
\makeatother


\begin{document}


\title{
Normal forms, moving frames, and differential invariants\\
for nondegenerate hypersurfaces in $\mathbb C^2$
}

\author{Peter J. Olver}
\address{School of Mathematics, University of Minnesota, Minneapolis, MN 55455, USA}
\email{olver@umn.edu}

\author{Masoud Sabzevari}
\address{Department of Mathematics,
Shahrekord University, 88186-34141, Shahrekord, IRAN and School of
Mathematics, Institute for Research in Fundamental Sciences (IPM), 19395-5746, Tehran, IRAN}
\email{sabzevari@ipm.ir}

\author{Francis Valiquette}
\address{Department of Mathematics,
Monmouth University, West Long Branch, NJ 07764,  USA}
\email{fvalique@monmouth.edu}

\date{\number\year-\number\month-\number\day}

\begin{abstract}
We use the method of equivariant moving frames to revisit the problem of normal forms and equivalence of nondegenerate real hypersurfaces $M \subset\mathbb C^2$  under the pseudo-group action of holomorphic transformations.  The moving frame recurrence formulae allow us to systematically and algorithmically recover the results of Chern and Moser for hypersurfaces that are either non-umbilic at  a point $\bp \in M$ or umbilic in an open neighborhood of it. In the former case, the coefficients of the normal form expansion, when expressed as functions of the jet of the hypersurface at the point, provide a complete system of functionally independent differential invariants that can be used to solve the equivalence problem. We prove that under a suitable genericity condition, the entire algebra of differential invariants for such hypersurfaces can be generated,  through the operators of invariant differentiation, by a single real differential invariant of order $7$. We then apply moving frames to construct new convergent normal forms for nondegenerate real hypersurfaces at singularly umbilic points, namely those umbilic points where the hypersurface is not identically umbilic around them.
\end{abstract}

\maketitle

\pagestyle{headings} \markright{Normal forms, moving frames, and differential invariants for hypersurfaces}

\numberwithin{equation}{section}

\section{Introduction}

The analysis of the local geometry of real hypersurfaces in the complex space $\mathbb C^n$ has its origins in the works of Poincar\'e, \cite{Poincare-1907}, and \'Elie Cartan, \cite{Cartan-1932}, which served to initiate the subject known as  Cauchy--Riemann (CR for short) geometry, \cite{Jacobowitz}.  The field received a major impetus with the seminal 1974 paper of Chern and Moser, \cite{Chern-Moser}, which applied two complementary methods to study the problem: normal forms based on Taylor expansions, and the Cartan equivalence method, \cite{Cartan-1932, Olver-1995}.  The two methods brought a different range of tools and results, and their precise interrelationship remains not entirely clear.  The Chern--Moser paper inspired many developments in CR geometry, continuing to this day; see, for example, \cite{Beloshapka-80,KKZ-17, Normal-Form,Webster-78}.

The aim of this paper is to reconcile the Cartan equivalence and normal form methods by conducting a careful analysis of hypersurfaces in $\mathbb C^2$.  In particular, we derive normal forms in the previously unstudied case of ``singularly umbilic points.''  The reconciliation will be accomplished through the method of equivariant moving frames that was developed by the first author and collaborators, \cite{Olver-Fels-99, Olver-Pohjanpelto-08}, as a reformulation of classical moving frames, \cite{Cartan-1935, Guggenheimer}, in a form that can be effectively and systematically applied to arbitrary finite-dimensional Lie group actions and, as required in the present situation, a large class of infinite-dimensional Lie pseudo-groups.  The method provides  powerful algorithmic tools that enable one to systematically determine the differential invariants, the invariant differential operators, the invariant differential forms, and  the complete structure of the associated differential invariant algebra.  This allows one to solve equivalence problems through the construction of differential invariant signatures, making it a compelling alternative to the Cartan equivalence method.   See also \cite{Arnaldsson,Valiquette-11, Valiquette-SIGMA} for further developments towards synthesizing and extending these two methods.  Since its inauguration, the method of equivariant moving frames has been used in an ever expanding range of interesting and novel applications throughout mathematics, physics, engineering, computer vision, and beyond, \cite{Olver-2015}.

The starting point for the construction of an equivariant moving frame is the choice of a cross-section to the prolonged (pseudo-)group orbits in the submanifold jet space.  In \cite{Olver-2018}, this construction was reinterpreted as the specification of a normal form for the submanifolds under the group action, thereby providing a bridge to the approaches found in Chern and Moser. The unnormalized or ``non-phantom'' coefficients in the normal form expansion provide a complete system of functionally independent differential invariants that can then be used to solve the associated equivalence problem. The choice of cross-section or, equivalently, normal form will specify or ``normalize'' expressions for the group parameters that serve to prescribe the equivariant moving frame map. More generally, through a recursive procedure, \cite{Olver-Valiquette}, one can introduce a succession of partial normal forms (partial cross-sections) which can be used to normalize more and more of the group parameters.  Substitution of the results into the prolonged transformation formulae can be regarded as an ``invariantization'' process that maps functions, differential forms, differential operators, etc., to their (partially) invariant counterparts.

The most important contribution of the method of equivariant moving frames are the remarkable {\it recurrence relations} that enable one to write the differentials of the  (partially) normalized invariants in terms of the (partially) normalized invariant horizontal one-forms and the (partially) normalized Maurer--Cartan forms associated with the pseudo-group, \cite{Olver-Pohjanpelto-05}, where the latter can be explicitly determined from the phantom recurrence formulae associated with the cross-section equations or, equivalently, the coefficients in the normal form that have been normalized to be constant.  Remarkably, these recurrence relations can be constructed from the cross-section by purely symbolic calculations, involving only linear algebra.  Notably, they do not require coordinate expressions for the differential invariants, the prolonged pseudo-group action, or even the moving frame itself!  Moreover, they are well behaved under partial normalization, retaining their basic form throughout the computations. The full details of the symbolic moving frame calculus will appear throughout our calculations.  It is worth re-emphasizing that, in contrast to the classical methods of solving equivalence problems and constructing normal forms, the equivariant moving frame method is a) completely systematic/algorithmic; b) straightforwardly implemented; and c) can be readily generalized to other problems of interest.  While the required calculations are facilitated by the use of symbolic manipulation software packages, only the basic formulae for the (infinitesimal) pseudo-group action is required as input, and they can be effected with minimal human intervention --- mostly just the specification of the cross-section --- during the computations.

Once a complete collection of differential invariants has been determined, the Cartan equivalence method enables us to solve both the equivalence problem for submanifolds under the prescribed transformation (pseudo-)group and to determine  the size of their (local) symmetry group.  To be precise, \cite{Olver-sg}, given a Lie pseudo-group acting on a space, a \emph{local symmetry} of a  submanifold $M$ at a point $\bp \in M$ is a pseudo-group transformation that maps an open neighborhood of $\bp$ to some other open subset of $M$. The set of local symmetries forms a groupoid that contains all global symmetries, but whose structure can vary over the submanifold.     In particular, the \emph{isotropy subgroup} at a point $\bp \in M$ is the group of local symmetries that fix $\bp$.  The isotropies form a bona fide group, whose structure can vary from point to point. For example, consider the rotation group ${\rm SO}(n)$ acting on $\R^n$; the isotropy subgroup at each point $0\neq\bp\in\mathbb R^n$ is isomorphic to ${\rm SO}(n-1)$, while at the origin it is all of ${\rm SO}(n)$.  Moreover, many open subsets of a sphere, $U \subset \mathbb S^{n-1} \subset \R^n$, have no global rotational symmetries, but each point $\bp \in U$ retains a full set of local rotational symmetries.

In the equivariant moving frames setting, a point $\bp \in M$ is called \emph{nonsingular} if its isotropy subgroup is discrete.  As shown in \cite{Olver-2000} (generalized to Lie pseudo-groups), this is equivalent to its jet of sufficiently high order belonging to the nonsingular subset of jet space where the pseudo-group acts locally freely. Nearby a nonsingular point, the number of  functionally independent  differential invariants equals 
the codimension of its local symmetry groupoid; see \cite[Theorem 14.26]{Olver-1995}, \cite[Theorem 5.17]{Olver-Fels-99}, \cite{Olver-sg} for details.
For example, if all the differential invariants are constant, then the $p$-dimensional nonsingular submanifold $M$ admits a 
$p$-dimensional local symmetry groupoid, and hence locally coincides with an orbit thereof.   For a finite-dimensional Lie transformation group, the \emph{totally singular submanifolds}, all of whose points are singular, are classified Lie algebraically in  \cite{Olver-2000}, although it is not known if this result can be generalized to infinite-dimensional pseudo-groups.

Turning to our specific problem, we introduce the local coordinates $z:=x+\i y, \ w:=u+\i v$ on the complex plane $\mathbb C^2$.  The \emph{Chern--Moser normal form} for a $3$-dimensional nondegenerate real hypersurface $M \subset \mathbb C^2$ at a point $\bp \in M$ is obtained by applying a sequence of holomorphic pseudo-group transformations that first map $\bp$ to the origin, and then fix the form of its convergent Taylor expansion
\begin{equation}\label{NF-CM-general}
v = z\oz+c_{42}(u)\, z^4\oz^2+c_{24}(u)\, z^2\oz^4+\sum_{j+k\geq 7\atop {{\rm min}(j,k)\geq 2}} c_{jk}(u)\, z^j \oz^k,
\end{equation}
for certain coefficients $c_{jk}(u)= \overline{c_{kj}(u)}$; see \cite[eq.\ (3.18)]{Chern-Moser}.  In particular, the Taylor coefficient $c_{42}(0)$ is a constant multiple of  the {\it Cartan curvature} of $M$ at $\bp$, \cite{Cartan-1932, Jacobowitz, Sabzevari-Merker-14}.

Chern and Moser call $\bp$ an {\it umbilic point} whenever the Cartan curvature vanishes there. In the {\it non-umbilic} case, further normalization enables us to uniquely fix the normal form \eqref{NF-CM-general} to be
\begin{equation}\label{NF-CM-non-umb}
v=z\oz+2 \,{\rm Re} \> \big\{z^4\oz^2(1+ \CMJ \,z+\i \CMK \,u)\big\}+\cdots\>,
\end{equation}
where $\CMJ, \CMK$ are seventh order differential invariants, with $\CMJ$ complex-valued, while $\CMK$ is real-valued; see \cite[eq. (3.20)]{Chern-Moser}.  The transformation $z \mapsto -z$ preserves the normal form \eqref{NF-CM-general}, while mapping $\CMJ \mapsto -\,\CMJ$.  For this reason, Chern and Moser designate $\CMJ^2$ as the primary invariant, but for us it is more convenient to work with $\CMJ$, keeping in mind the inherent sign ambiguity\footnote{This is similar to the ambiguity in the curvature invariant $\kappa $ of a curve $C$ in the Euclidean plane $\R^2$, whose sign changes under a $180^{\circ}$ rotation, and hence depends on the curve's orientation.}.
For such hypersurfaces, there remains only a discrete isotropy subgroup,  which is either trivial or contains exactly two elements, \cite{Beloshapka-80}.  As in \cite{Olver-2018}, the non-phantom Taylor coefficients in \eqref{NF-CM-non-umb} provide the complete system of functionally independent differential invariants (modulo sign changes) for non-umbilic hypersurfaces under holomorphic transformations.

\medskip
{\it Remark\/}: \ 
As noted by Chern and Moser, the case of hypersurfaces in $\C^2$ is quite different from that of $\C^n$ for $n\geq 3$; in particular the latter cases produce differential invariants of lower orders.
\medskip

We will apply the equivariant moving frame recurrence formulae to prove that, on a general non-umbilic hypersurface the entire algebra of differential invariants can be obtained by repeated invariant differentiation of the four real-valued invariants ${\rm Re} \>\CMJ, {\rm Im} \>\CMJ,\CMK, \CML$, where the latter is the eighth order differential invariant appearing as the coefficient of $z^4\oz^4$ in the Chern--Moser normal form \eqref{NF-CM-non-umb}. Even better, if the hypersurface is ``$\CMK$-nondegenerate'', as formulated in Definition \ref{CMKnondegenerate}, then its differential invariant algebra  can, in fact, be generated by the single real-valued Chern--Moser invariant $\CMK$.

On the other hand, if the hypersurface $M$ is umbilic {\it in a neighborhood} of a point $\bp$, then it is locally biholomorphically equivalent to the Heisenberg  sphere $\mathbb H$, which is defined by $v=z\oz$.  This  sphere admits an $8$-dimensional holomorphic symmetry group generated by the vector fields
\begin{equation}\label{G8}
\aligned
&\vv_1=(w+2\i z^2)\,\partial_z+2\i zw\,\partial_w, \quad
\vv_2=(\i w+2\,z^2)\,\partial_z+2\,zw\,\partial_w,\quad
\vv_3=zw\,\partial_z+w^2\,\partial_w,\\
&\vv_4=z\,\partial_z+2\,w\,\partial_w,\quad
\vv_5=\i z\,\partial_z.\quad
\vv_6=\partial_z+2\i z\,\partial_w,\quad
\vv_7=\i \partial_z + 2 z\,\partial_w.\quad
\vv_8=\partial_w,
\endaligned
\end{equation}
which span a real Lie algebra isomorphic to $\frak{sl}(3,\R)$; see, e.g., \cite[Proposition 2.1]{Merker-Sabzevari-CEJM}.  At each point on the sphere, there is a $5$-dimensional isotropy subalgebra which is isomorphic to its maximally solvable Borel subalgebra; for example, the first $5$ vector fields span the isotropy subalgebra  at the origin. Moreover, every totally singular nondegenerate hypersurface is locally equivalent to the umbilic Heisenberg sphere.

In addition to these two known scenarios, there exists another case which has garnered less attention in the literature but is worth investigating in detail. A point $\bp \in M$ is {\it singularly umbilic}  if the hypersurface is not identically umbilic on any local neighborhood thereof.
Another key contribution of this paper is to apply the equivariant moving frame method to construct a new normal form for hypersurfaces at such points.
We say that a singularly umbilic point $\bp \in M$ has {\it umbilic order}  $n_0\geq 7$ if this is the minimal order\footnote{For this purpose, we ignore the order $2$ term $z \oz$.} at which a non-zero term of that order appears in the partial normal form \eqref{NF-CM-general}. The classification problem then splits into three further subcases. Accordingly, we call $\bp$ {\it circular} when the partial normal form \eqref{NF-CM-general} only contains terms involving $c_{kk}(u)z^k \oz^k$, i.e., the sum is only over $j=k$; in this case, $n_0\geq 8$.  In the non-circular case, we call $\bp$ {\it semi-circular} whenever in its corresponding partial normal form \eqref{NF-CM-general},  all order $n_0$ monomials are of the form $(z\oz)^k u^\ell$. Otherwise, we say that $\bp$ is a {\it generic}\footnote{The reader should not confuse our use of this term with the standard definition of generic submanifolds in CR geometry, \cite{Jacobowitz}.} singularly umbilic point.  In particular, apart from the fully umbilic Heisenberg sphere, the only points on a nondegenerate hypersurface with continuous isotropy, and hence singular in the above sense, are the circular singularly umbilic points.  By the definition of singularly umbilic, any neighborhood of a circular point must contain non-umbilic points, which are necessarily nonsingular, thus proving our claim that every totally singular hypersurface is locally equivalent to the Heisenberg sphere.  This moreover implies that the local symmetry group of a nondegenerate hypersurface $M$ can have dimensions $0$ (discrete), $1$, $2$, $3$, or $8$ while, as in \cite{Beloshapka-80}, the isotropy subgroup at a point can have dimensions $0$ (discrete), $1$, or $5$.  In both, the largest dimensions are achieved only on fully umbilic (totally singular) hypersurfaces.

The structure of the paper is given as follows.  In Section \ref{sec-prel} we introduce the preliminary materials necessary for constructing normal forms of hypersurfaces in $\mathbb{C}^2$ under the pseudo-group action of biholomorphic transformations.  In Section \ref{sec-NF-general-comp} we begin the implementation of the normalization process using the {\it recurrence relations} for the (partially) normalized invariants. We perform the normalizations up to order six, at which point the problem splits into two branches depending on whether the point is umbilic or non-umbilic. Sections \ref{non-umb} and \ref{umb} re-establish the Chern--Moser normal forms for, respectively, the non-umbilic and umbilic cases.   Section \ref{non-umb} also contains the proof of our result concerning the structure and generators of the differential invariant algebra associated with non-umbilic hypersurfaces. Section \ref{sec-equi-rigid} contains results on the equivalence and rigidity properties of these hypersurfaces.
Further new results are presented starting in Section \ref{sec-semi-umb}, where we consider nondegenerate hypersurfaces $M$ that are singularly umbilic at a point $\bp$. The three cases --- generic, semi-circular, and circular --- are treated in Sections \ref{sec-generic-semi-umb}, \ref{sec-semi-circular}, and \ref{sec-circular}, respectively. Finally, in Section \ref{sec-iso} we will briefly discuss the symmetry and isotropy groups associated with  nondegenerate hypersurfaces.

Throughout, we set  $\N = \{ 0 < n \in \mathbb{Z}\}$ and $\N_0 = \{ 0 \leq n \in \mathbb{Z}\}$.  In general, we will work purely symbolically and do not attempt to find the local coordinate expressions for the differential invariants, the invariant differential operators and forms, the moving frame, and so on.  In principle, the local coordinate expressions can be constructed using the non-symbolic equivariant moving frame calculus, but the required computations are daunting.

\section{Preliminary materials}\label{sec-prel}

In this section we set up the local equivalence problem for real hypersurfaces $M \subset\mathbb{C}^2$ under holomorphic transformations in preparation for applying the method of equivariant moving frames.

\subsection{The pseudo-group of holomorphic transformations}

As above, we will use $z= x +\i y, \ w=u+\i v$ to denote the standard coordinates on the complex plane $\mathbb{C}^2$.  In our analysis of real hypersurfaces, it will be convenient to employ $z, \oz,u,v$ as the basic coordinates.

The Lie pseudo-group $\G$ consists of holomorphic transformations of $\mathbb{C}^2$, and hence is of the form
\begin{equation}\label{G}
\varphi\colon (z, w)\ \longmapsto \ {\bigl(\,Z :=F(z, u, v), \  W= U + \i V := G(z, \oz, u, v)+ \i H(z, \oz, u, v)\,\bigr),}
\end{equation}
so that $F$ is complex-valued, while $G,H$ are real-valued analytic functions.
The transformations in $\G$ are characterized by the basic determining equations, \cite[$\S3$]{Normal-Form},
\begin{equation}\label{eq: determining equations}
\begin{gathered}
{F_{\oz} = \oF_z = 0,\qquad F_v = \i F_u,\qquad \oF_v = -\i \oF_u,}\\
{H_z = -\i G_z,\qquad H_{\oz}=\i G_{\oz},\qquad H_u=-G_v,\qquad H_v = G_u.}
\end{gathered}
\end{equation}
We use $g\n  = \jet_n \varphi= (F\n,\oF\n,H\n,G\n)$ to denote the $n$-th order jet of the transformation \eqref G, whose components are the partial derivatives of order $0 \leq k \leq n$ of the functions with respect to their arguments.
Repeated differentiation of the involutive first order determining equations \eqref{eq: determining equations} provides the complete system of constraints on the jets of pseudo-group transformations.

Let $\mathfrak{g}$ denotes the collection of all local infinitesimal generators of $\G$.  Note that $\mathfrak{g}$ is closed under the Lie bracket operation, but does not form a Lie algebra in the usual sense since the Lie bracket of two vector fields is only defined on the intersection of their domain of definition.  By linearizing the determining equations \eqref{eq: determining equations} at the identity transformation, cf. \cite{Olver-1995, Olver-Pohjanpelto-05}, we see that a vector field
\begin{equation}\label{eq: v}
\vv = \xi(z,u,v)\pp{}{z} + \overline{\xi(z,u,v)}\pp{}{\oz} + \eta(z,\oz,u,v)\pp{}{u} + \phi(z,\oz,u,v)\pp{}{v}
\end{equation}
belongs to $\mathfrak{g}$ if and only if its components satisfy the \emph{infinitesimal determining equations}
\begin{subequations}\label{eq: infinitesimal determining equations}
\begin{equation}\label{eq: order 1 infinitesimal determining equations}
\begin{gathered}
\xi_{\oz}=\oxi_z=0,\qquad \xi_v = \i \xi_u,\qquad \oxi_v=-\i \oxi_u, \\
\phi_z = -\i \eta_z,\qquad \phi_{\oz} = \i \eta_{\oz},\qquad \phi_u = - \eta_v,\qquad \phi_v=\eta_u,
\end{gathered}
\end{equation}
where the subscripts indicate partial derivatives with respect to the indicated variables.
Differentiating the determining equations \eqref{eq: order 1 infinitesimal determining equations} we obtain the following relations among the second order vector field jets
\begin{equation}\label{eq: 2nd order jet relations}
\begin{aligned}
\xi_{\oz a} &= 0, &\qquad \xi_{zv} &= \i \xi_{zu},& \qquad \xi_{uv} &= \i \xi_{uu},&\qquad \xi_{vv} &= -\xi_{uu},\\
\oxi_{z a} &= 0, &\qquad  \oxi_{\oz v} &= -\i \oxi_{\oz u}, &\qquad \oxi_{uv} &= -\i \oxi_{uu},&\qquad \oxi_{vv} &= -\oxi_{uu},\\
\eta_{z\oz} &= \phi_{z\oz}=0, &\qquad \eta_{zv}&=\i \eta_{zu},&\qquad \eta_{\oz v} &= -\i \eta_{\oz u},  &\qquad \eta_{vv} &=-\eta_{uu},\\
\phi_{z a} &= -\i \eta_{z a},&\qquad \phi_{\oz a} &= \i \eta_{\oz a}, &\qquad \phi_{u a} &= -\eta_{v a}, &\qquad \phi_{v a} &= \eta_{u a},
\end{aligned}
\end{equation}
\end{subequations}
where $a \in \{z,\oz,u,v\}$.  Relationships among the higher order vector field jets $\zeta\n =(\xi\n,\oxi\n,\eta\n,\phi\n)$ are obtained by further differentiation of the second order determining equations \eqref{eq: 2nd order jet relations}.

As in \cite{Olver-Pohjanpelto-05}, let $\bmu^{(\infty)} =(\mu_A, \omu_A, \mcu_A, \mcv_A)$ denote the Maurer--Cartan forms for the pseudo-group $\D$ of local diffeomorphisms of $\mathbb{C}^2$, where $A=(a_1,\ldots, a_k)$, with $k \geq 0$, is a symmetric multi-index with $a_\nu \in \{Z,\oZ,U,V\}$. (We will not need their explicit formulae here.) Restricting them to the sub-pseudo-group $\G$ produces its Maurer--Cartan forms, for which we use the same notation, but now they are no longer linearly independent.   According to the structure Theorem 6.1 in \cite{Olver-Pohjanpelto-05}, the infinitesimal determining equations \eqref{eq: infinitesimal determining equations} yield, after the \emph{invariantization} substitutions
\begin{equation}\label{eq: substitutions}
\xi\longmapsto \mu,\quad \oxi\longmapsto \omu,\quad \eta\longmapsto \mcu,\quad \phi\longmapsto \mcv,\quad z\longmapsto Z,\quad \oz\longmapsto \oZ,\quad u\longmapsto U,\quad v\longmapsto V,
\end{equation}
the complete system of linear relations among the Maurer--Cartan forms on the pseudo-group $\G$.  In particular, the first and second order Maurer--Cartan forms satisfy
\begin{equation}\label{eq: MC relations}
\begin{aligned}
& \mu_{\oZ} = 0,&\qquad &\omu_Z = 0,&\qquad &\mu_V=\i \mu_U,&\qquad &\omu_V = -\i \omu_U,\\
&\mcv_Z = -\i \mcu_Z,&\qquad &\mcv_{\oZ}=\i \mcu_{\oZ},&\qquad&\mcv_U = - \mcu_V,&\qquad &\mcv_V = \mcu_U,\\
&\mu_{\oZ a} = 0,&\qquad &\mu_{ZV} = \i \mu_{ZU},&\qquad &\mu_{UV} = \i \mu_{UU},&\qquad &\mu_{VV}=-\mu_{UU},\\
&\omu_{Z a} = 0,&\qquad &\omu_{\oZ V} = -\i \omu_{\oZ U},&\qquad& \omu_{UV} =-\i \omu_{UU},&\qquad& \omu_{VV} = -\omu_{UU},\\
&\mcu_{Z\oZ}=\mcv_{Z\oZ}=0,&\qquad& \mcu_{ZV} = \i \mcu_{ZU},&\qquad& \mcu_{\oZ V} = -\i \mcu_{\oZ U},&\qquad& \mcu_{VV} = - \mcu_{UU},\\
&\mcv_{Z a} = -\i \mcu_{Z a},&\qquad& \mcv_{\oZ a} = \i \mcu_{\oZ a},&\qquad &\mcv_{U a} = - \mcu_{V a},&\qquad& \mcv_{V a} = \mcu_{U a},
\end{aligned}
\end{equation}
where $a \in \{Z, \oZ, U, V\}$, and similarly for their higher order counterparts. It follows that a basis of Maurer--Cartan forms for $\mathcal G$ is provided by
\[
\mu_{Z^kU^\ell},\qquad \omu_{\oZ{}^k U^\ell},\qquad \mcu_{U^k V}, \qquad \mcu_{Z^k U^\ell},\qquad \mcu_{\oZ{}^k U^\ell},\qquad \mcv,\qquad \text{where}\qquad k,\, \ell \geq 0.
\]
The structure equations for the holomorphic pseudo-group $\G$ can be deduced by the general results in \cite{Olver-Pohjanpelto-05}, but will also not be needed in what follows.

\subsection{Holomorphic action on nodegenerate hypersurfaces}

We are interested in the local CR geometry of real hypersurfaces $M \subset\mathbb C^2$.  As in Chern and Moser, \cite{Chern-Moser}, we assume throughout that $M$ is real analytic, although almost all of our results (apart from the convergence properties of the normal form series) are equally valid for smooth, meaning $C^ \infty $ hypersurfaces. We call two hypersurfaces $M, \widetilde M$ (locally) \emph{congruent} if there is a holomorphic map $g \in \G$ such that (locally) $\widetilde M = g \cdot M$.

At any point $\bp \in M$, one can choose local coordinates $z, w=u+\i v$, in terms of which the hypersurface is represented as the graph of the real-analytic function
\[
v=f(z,\oz, u).
\]
We assume throughout that $M$ is everywhere (Levi) \emph{nondegenerate}, meaning that the corresponding \emph{Levi form} $v_{z\oz} = \partial ^2 f/\partial z \partial \oz$ is nonzero.

For $0 \leq n \leq \infty $, let $\Jn$ denote the $n$-th order jet bundle of hypersurfaces $M \subset \C^2$.
The action of the pseudo-group $\G$ on $\mathbb{C}^2$ induces an action on hypersurfaces, and hence on their jets, which is known as the \emph{prolonged action}, \cite{Olver-1995}. In terms of the usual jet coordinates $v\n =(v, v_z, v_{\oz}, v_u, \ldots, v_J, \ldots\>) \in \Jn$, where $J$ is a symmetric multi-index of order $\#J\leq n$ whose entries $j_ \nu \in \{z,\oz,u\}$, that represent the partial derivatives of $v$ of order $\leq n$, the corresponding transformed jet coordinates are labelled with capitalized letters: $V\n = (V,V_Z, V_{\oZ}, V_U, \ldots, V_J, \ldots\>)$, where now the entries of the multi-index $J$ are $j_ \nu \in \{Z,\oZ,U\}$, in accordance with Cartan's convention, cf.~\cite{Olver-Pohjanpelto-05}.  The rather complicated explicit expressions for the prolonged pseudo-group action follow from the usual rules of implicit differentiation, \cite{Olver-Pohjanpelto-08}, but will not be used in this paper\footnote{They would be required if one wishes to compute the explicit formulae for the normalized differential invariants through the equivariant moving frame calculus, but here we choose to work purely symbolically.}.

\subsection{Moving frames}

The method of equivariant moving frames for finite-dimensional Lie group actions was introduced in \cite{Olver-Fels-99} and extended to infinite-dimensional pseudo-groups in \cite{Olver-Pohjanpelto-08}.  In outline, it proceeds via the following algorithm. One begins by prolonging the (pseudo-)group action to the relevant submanifold jet bundles; the formulae for the transformed derivatives (jet coordinates) are known as the \emph{lifted differential invariants} since they can be viewed as invariant functions on a certain ``lifted bundle.''  At each successive jet order, one chooses a cross-section to the prolonged action; in practice, this amounts to normalizing one or more of the lifted invariants by setting them equal to a constant, typically either $0$ or $1$, although other choices may at times be required and/or more convenient\footnote{This describes what is known as a ``coordinate cross-section''; more general cross-sections are rarely used and will not concern us here.}.  Each normalization equation is then solved for one of the pseudo-group parameters, and the result is a partial moving frame, which can be interpreted as an equivariant subbundle of the lifted bundle.  One then substitutes the resulting formulae for the pseudo-group parameters back into the other lifted invariants, producing the \emph{partially normalized lifted invariants}.  Those that were used to specify the cross-section become, of course, their chosen constants, and are referred to as the \emph{phantom invariants}.  Any non-constant partially normalized invariant that no longer depends upon any of the pseudo-group parameters is a differential invariant in the usual sense.  One then iterates this procedure by recursively increasing the jet (derivative) order, until (in favorable cases) one has eliminated all the pseudo-group parameters; the resulting  non-phantom normalized invariants form a complete system of functionally independent differential invariants for the action.  For details on the recursive procedure, see \cite{Olver-2011,Olver-Valiquette}.

Now, the formulae for the prolonged pseudo-group transformations, i.e., the lifted invariants, can become exceedingly complicated, and the practical implementation of the moving frame algorithm may be beyond even  powerful modern symbolic computer algebra systems.  (Another stumbling block is the appearance of rational algebraic functions, which are still not handled particularly well by current software.)  Nevertheless, as can be seen in the cited references, the method has been successfully applied to a very broad range of actions.  There is, however, an alternative, purely symbolic approach that relies purely on linear algebra, and hence can be easily implemented in very challenging situations, including the problem under consideration here. While it does not produce the explicit formulae for the differential invariants and other invariant quantities, it does lead to formulae expressing their interrelationships, as well as the associated normal form, and thus provides significant insight into the overall structure of the problem.  The symbolic method relies on the remarkable \emph{recurrence formulae} that express the differentials of a lifted invariant as a lifted invariant linear combination of two types of invariant differential forms --- the lifted invariant horizontal one-forms, which are obtained by applying the pseudo-group transformations to the coordinate forms $dx^i$ for the independent variables $x^i$, and the Maurer--Cartan forms associated with the action; see Theorem \ref{rr} for details.  The recurrence relations are universal, relying only on the standard and relatively simple formula for the prolonged infinitesimal generators of the action, \cite{Olver-1995}, and do not require knowledge of the actual formulae for the lifted invariants, or the lifted invariant forms, or even the moving frame itself.  Inserting the selected normalization equation for a phantom lifted invariant into its associated recurrence relation produces a linear equation that can be solved for one of the Maurer--Cartan forms as a lifted invariant linear combination of the invariant horizontal forms and the other Maurer--Cartan forms, which thereby determines its partial normalization. Moreover, the allowed choices of normalization constant (zero, nonzero, \dots) are dictated by the structure of the recurrence formula, as the resulting equation must be compatible and solvable for one of the Maurer--Cartan forms.  The resulting normalization of a Maurer--Cartan form corresponds directly to the normalization of a group parameter in the non-symbolic version of the algorithm.

The recurrence formulae retain their structure under the partial normalization process described above, and this allows one to recursively continue the symbolic algorithm, finally leading (in favorable cases) to formulae for all the fully normalized Maurer--Cartan forms in terms of the differential invariants and fully invariant horizontal forms.  The latter one-forms form an invariant horizontal coframe with dual invariant differential operators that map differential invariants to differential invariants.   The remaining fully normalized recurrence formulae form a complete system of identities (differential syzygies) that express the invariant derivatives of the normalized differential invariants as functions thereof, which thus completely determines the structure of the underlying algebra of differential invariants.  They consequently enable one to symbolically determine a system of fundamental differential invariants that generate the entire algebra through invariant differentiation.  Details of the process will become clear in the context of our chosen example.

A very simple example to keep in mind is the case of plane curves under the action of the special Euclidean group (translations and rotations).  The basic normalized differential invariant of lowest order is the curvature $\kappa $. The invariant horizontal form is the arc length $ds$ and the dual invariant differential operator is the arc length derivative.  The higher order differential invariants are then just the successive derivatives of curvature with respect to arc length: $\kappa , \kappa _s, \kappa _{ss}, \ldots \>$.  The recurrence formulae enable one to write the normalized differential invariants obtained through the moving frame process in terms of $\kappa $ and its arc length derivatives; see, for example, \cite{Olver-Fels-99,Olver-2015} for details on the calculations in this simple classic example.

Let us now provide the details of the symbolic moving frame algorithm, in the specific context of  real hypersurfaces $M \subset \C^2$.  For $0 \leq n \leq \infty $, the submanifold jets $\bz\n = (z,\oz,u,v\n)\in \Jn$ and the pseudo-group jets $g\n\in \G\n$, with source at $\bz=(z,\oz,u,v)$, parametrize the \emph{$n$-th order lifted bundle} $\B\n$.  There is a natural \emph{right action} of the holomorphic pseudo-group $\G$ on $\B\n$ given by combining the prolonged action on hypersurface jets with the action on $\G\n$ given by right composition of pseudo-group jets
\begin{equation}\label{eq: right action}
R_\psi(\bz\n,\jet_n\varphi|_{\bz}) = (\bZ\n,\jet_n(\varphi\circ \psi^{-1})|_{\psi(\bz)})\quad\text{for any}\quad \varphi,\,\psi\, \in\, \G,
\end{equation}
where the composition $\varphi\circ \psi^{-1}$ is defined.  When viewed as functions on the lifted bundle $\B\n$, the transformed quantities $\bZ\n = (Z,\oZ,U,V\n)$ resulting from the coordinate expressions for the prolonged pseudo-group action on $\Jn$ are invariant under the right bundle action \eqref{eq: right action}, and are hence referred to as \emph{lifted} (\emph{differential}) \emph{invariants}. The lifting procedure can be applied to an arbitrary differential function, and defines an \emph{invariantization} process
\[
\iota\bigl[ F(z,\oz,u,v\n) \bigr] = F(Z,\oZ,U,V\n),
\]
that maps differential functions to their lifted invariant counterparts, whose action on the basic jet coordinates is given by\footnote{The multi-index $J$ on $v_J$ contains all lower case letters, in $\{z,\oz,u\}$, whereas on $V_J$ it denotes the multi-index with the corresponding capital letters, in $\{Z,\oZ,U\}$, and should perhaps be denoted $\iota(J)$. However, since the type of a multi-index subscript is always clear from context, we choose not to unnecessarily clutter the notation.}
\begin{equation}\label{eq: iota}
\iota(z) = Z,\qquad
\iota(\oz) = \oZ,\qquad
\iota(u) = U,\qquad
\iota(v_J) = V_J.
\end{equation}
One can similarly lift a differential form on the hypersurface jet bundle $\Jn$  by applying the prolonged pseudo-group transformations, the result being an invariant differential forms on the lifted bundle $\B\n$, \cite{Olver-Pohjanpelto-08}.  In particular, the lifts of the basic horizontal one-forms $dz$, $d\oz$, $du$, form the lifted {\it invariant horizontal coframe}, and are denoted
\begin{equation}\label{eq: omega}
\iota(dz) = \omega^Z,\qquad
\iota(d\oz) = \omega^{\oZ}=\overline{\omega^Z},\qquad
\iota(du) = \omega^U.
\end{equation}
Their coordinate expressions are easily found, but  will not be needed here.

\begin{Remark}
In the ensuing calculations, we will always, for simplicity, omit contact forms from all of our formulae, and so, technically, the differential on the hypersurface jet bundle refers to the horizontal differential; see \cite{Kogan-Olver-2003} for details. The contact forms play no role here, since they vanish on all hypersurfaces, but they are of importance in other contexts, in particular in the study of invariant variational problems.
\end{Remark}

The recursive moving frame normalization process can be formalized as follows, \cite{Olver-2011,Olver-Valiquette,Valiquette-SIGMA}.

\begin{Definition}
A \emph{partial right moving frame} of  order $n$ is a right-invariant local subbundle $\hB\n\subset \B\n$, meaning that $R_\psi(\hB\n) \subset \hB\n$ for all $\psi \in \G$ where the right action \eqref{eq: right action} is defined. If the subbundle $\hB\n$ forms the graph of a right-invariant section of $\B\n$, it defines an equivariant moving frame.
\end{Definition}

Partial moving frames are obtained by choosing a cross-section $\mathcal{K}\n$ to the prolonged action at some order $n$, and solving the corresponding \emph{normalization equations} for the pseudo-group parameters, \cite{Olver-Pohjanpelto-08}.
Assuming the prolonged action is regular, if it is also (locally) free, that is, the isotropy group at a submanifold jet $\bz\n$ is trivial or discrete, then the notion of a partial moving frame reduces to the usual definition of an equivariant  moving frame.

A (partial) moving frame induces its own invariantization process $\widehat \iota$ that maps differential functions $F$ and, more generally, differential forms $\omega $ on the hypersurface jet space to their (partially) invariant counterparts on $\hB\n$, with
$$\widehat \iota\,(\omega ) = (\hrho\n)^* \iota( \omega ),$$
where $\hrho\n\colon \hB\n \hookrightarrow \B\n$ denotes the inclusion map.
In practice, the invariantization process associated with a partial moving frame  simply amounts to replacing the pseudo-group parameters that appear in the original lifted invariants by their (partial) moving frame expressions that are obtained by solving the normalization equations; the resulting invariantized functions and forms can depend upon any remaining unnormalized pseudo-group parameters.

In the sequel, since all relations among differential functions and differential forms are preserved under pull-back, we will omit explicit reference to the pull-back map $(\hrho\n)^*$ and simply use $\iota$ for all the invariantization processes in order to streamline the notation.  Thus, at each stage in the computation, $Z, \oZ, U, V, V_J, \omega^Z,\omega^{\oZ},\omega^U$ will refer to the current partially normalized lifted invariants and forms. Once all the pseudo-group parameters have been normalized (when possible), these will become the sought-for differential invariants and invariant differential forms.

\subsection{Recurrence relations}

One of the most fundamental results in the theory of equivariant moving frames is the {\it universal recurrence formula}, \cite[Theorem 25]{Olver-Pohjanpelto-08}, which unlocks the structure of the algebra of (lifted) differential invariants. To write the formula, recall that the prolongation of the vector field \eqref{eq: v} to the submanifold jet bundle $\Jn$ is the vector field
\begin{equation}
\label{v-infty}
\vv\n = \xi \frac{\partial}{\partial z} + \oxi \frac{\partial}{\partial \oz} + \eta \frac{\partial}{\partial u} + \sum_{0 \leq \# J \leq n} \phi^J\frac{\partial}{\partial v_J},
\end{equation}
whose coefficients are given by the standard prolongation formula,  \cite{Olver-1995},
\begin{subequations}
\begin{equation}\label{eq: prolongation formula}
\phi^J(\bz\n,\zeta^{(n)}) = D_J(\phi - v_z\,\xi - v_{\oz}\,\oxi  - v_u \eta) + v_{J, z}\,\xi  + v_{J, \oz}\,\oxi  + v_{J, u}\,\eta ,
\end{equation}
and where $D_J = D_{j_1}\cdots D_{j_n}$ denotes the total derivative operator corresponding to the symmetric multi-index $J=(j_1,\ldots,j_n)$, with $j_\nu \in \{z,\oz,u\}$, of order with $\#J = n$.  One can, alternatively, use the recursive version
\begin{equation}\label{eq: prolongation formula 2}
\phi^{J,k} = D_k \phi^J - v_{J,z} \, D_k\xi  - v_{J,\oz}\, D_k\oxi  - v_{J,u}\, D_k\eta.
\end{equation}
\end{subequations}
For example, taking into account the infinitesimal determining equations \eqref{eq: order 1 infinitesimal determining equations}, the order 1 prolonged vector field coefficients are
\begin{equation}\label{eq: order 1 prolonged coefficients}
\begin{aligned}
\phi^z &= -v_z \,\xi_z - \i v_z^2\,\xi_u + \i v_z v_{\oz} \,\oxi_u - (\i+v_u)\,\eta_z + v_z\,\eta_u-v_zv_u\,\eta_v,
\\
\phi^{\oz} &= -v_{\oz}\, \oxi_{\oz} + \i v_{\oz}^2\, \oxi_u - \i v_z v_{\oz}\, \xi_u + (\i - v_u)\,\eta_{\oz} + v_{\oz}\, \eta_u -  v_{\oz}v_u\, \eta_v,
\\
\phi^u &= -v_z\,(1+\i v_u)\,\xi_u - v_{\oz}\,(1-\i v_u)\,\oxi_u-(1+v_u^2)\,\eta_v.
\end{aligned}
\end{equation}

The universal recurrence formula can then be written as follows.

\begin{Theorem}\label{rr}
If $F(z,\oz,u,v\n)$ is any differential function, then
\begin{equation}\label{rrF}
\begin{aligned}
d\, \iota(F) = \iota(dF) + \iota\bigl[\vv\n(F)\bigr].
\end{aligned}
\end{equation}
\end{Theorem}

The final ``correction'' term in \eqref{rrF} is obtained by invariantization of the infinitesimal generator coefficient \eqref{eq: prolongation formula} --- that is, by applying the substitutions \eqref{eq: substitutions} to the submanifold and vector field jets appearing therein.
In particular, setting $F = z, \oz, u, v_J$ in turn, we deduce the complete system of recurrence relations for the basic lifted invariants
\begin{equation}\label{rec-rel}
dZ = \omega^Z + \mu, \qquad
d\oZ = \omega^{\oZ} + \omu,\qquad
dU = \omega^U + \mcu, \qquad
dV_J=\varpi_J+\phi^J(\bZ\n,\bmu\n),
\end{equation}
where $\omega^Z,\omega^{\oZ},\omega^U$ are the lifted invariant horizontal forms \eqref{eq: omega}, while (again omitting contact forms)
\begin{equation}\label{varpi}
\varpi_J = \iota(dv_J) = \iota\bigl(v_{J,z}\,dz + v_{J,\oz}\,d\oz + v_{J,u}\,du\bigr) = V_{J,Z}\,\omega^Z + V_{J,\oZ}\, \omega^{\oZ} + V_{J,U}\,\omega^U.
\end{equation}
Pulling-back the recurrence relations \eqref{rec-rel} by a (partial) moving frame $\hrho\n$ yields the corresponding recurrence relations for the (partially) normalized invariants involving the (partially) normalized horizontal coframe and the (partially) normalized Maurer--Cartan forms.  As noted above, since pull-backs do not alter the basic equations, we will omit explicitly writing them from here on.

\begin{Remark}
In a number of our calculations, we will only need to know the group components in the recurrence formulae.  We thus write $\mu \equiv \mcu $ if they differ by a linear combination of the horizontal forms $\omega^Z,\omega^{\oZ},\omega^U$.  For example, the last recurrence relation in \eqref{rec-rel} can be written as
$dV_J \equiv \phi^J(\bZ\n,\bmu\n)$.
\end{Remark}

\begin{Remark}
In the universal recurrence formula \eqref{rrF}, the differential function $F$ can be replaced by any differential form on the hypersurface jet space, where $\vv\n$ acts via Lie differentiation and one inserts a wedge product between the invariantized vector field coefficients and the resulting differential forms in the correction term.  For example, taking $F$ to be the basic horizontal forms $dz, d\oz,du$, and noting that $dF = 0 $ in each case, we deduce the following recurrence formulae for the invariant horizontal one-forms \eqref{eq: omega}:
\begin{equation}\label{rrihf}
\begin{aligned}
d\omega^Z &= \mu_Z \wedge \omega^Z + \mu_{\oZ} \wedge \omega^{\oZ} + \mu_U \wedge \omega^U, \\
d\omega^{\oZ} &= \omu_Z \wedge \omega^Z + \omu_{\oZ} \wedge \omega^{\oZ} + \omu_U \wedge \omega^U, \\
d\omega^U &= \mcu _Z \wedge \omega^Z + \mcu _{\oZ} \wedge \omega^{\oZ} + \mcu _U \wedge \omega^U. \\
\end{aligned}
\end{equation}
\end{Remark}

The following formula, which will be of use later, follows from an application of the Leibniz rule to the prolongation formula \eqref{eq: prolongation formula}.  Specifically,
$$D_J \phi = \phi _J + \sum_K {J \choose K} v_K \,\phi _{J\setminus K,v} + \>\cdots,\qquad
D_J (v_z\,\xi)  = \sum_K {J \choose K} v_{K,z}\, \xi _{J\setminus K} + \>\cdots,$$
and similarly for the terms coming from $v_{\oz}\,\oxi$ and $v_u \eta$.  Here ${J \choose K}$ denotes the multinomial coefficient corresponding to the multi-indices $J,K$, and the omitted terms have the form of a positive integer multiplying one of these monomials
 \begin{equation}
\label{monomials}
\aligned
&v_{J_1} \cdots v_{J_{l-2}}v_{J_{l-1}} \phi _{J_l,v^{l-1}},\qquad  &
v_{J_1} \cdots v_{J_{l-2}}v_{J_{l-1},z} \xi _{J_l,v^{l-2}}, \\
&v_{J_1} \cdots v_{J_{l-2}}v_{J_{l-1},\oz} \oxi _{J_l,v^{l-2}}, &
v_{J_1} \cdots v_{J_{l-2}}v_{J_{l-1},u} \eta _{J_l,v^{l-2}},
\endaligned
\end{equation}
where $l \geq 3$, \ $\emptyset \ne J_\kappa \subsetneq J$  for $1\leq \kappa < l$, and $(J_1 \cdots J_l) = J$.  Using the determining equations \eqref{eq: infinitesimal determining equations} and their derivatives in the resulting formulae yields the prolonged vector field coefficient
 \begin{equation}
\label{phijkl}
\aligned
\phi^{z^j\oz^k u^{\ell}}&=-v_{z^j\oz^k u^{\ell}}\,\bigl[j\,\xi_{z}+k\,\overline\xi_{{\oz}}+(\ell-1)\, \eta_u\bigr]-\bigl[\ell\,v_{z^{j+1}\oz^k u^{\ell-1}} + \i N_{jk\ell}\bigr]\xi_{u}\\
&\qquad {} -\bigl[\ell\,v_{z^{j}\oz^{k+1} u^{\ell-1}} - \i \overline{N_{kj\ell}}\,\bigr]\overline\xi_{u} -\ell\, v_{z^j\oz^k u^{\ell-1}}\left[\,j\,\xi_{zu}+k\,\overline\xi_{{\oz}u}+\frac{\ell-3}2\,\eta_{uu}\,\right] \\
&\qquad{} - \i P_{jk\ell}\,\xi_{zu} + \i \overline{P_{kj\ell}}\;\overline\xi_{{\oz}u}+ Q_{jk\ell}\,\eta_{uu} + \>\cdots,
\endaligned
\end{equation}
whenever $j+k+\ell \geq 2$, and where the omitted terms depend on the other derivatives of $\xi,\overline\xi,\eta$ that are not displayed\footnote{Note that when $\ell = 0$,  terms involving $u^{\ell-1}$ are undefined, but have a vanishing coefficient and so do not appear.}.  The coefficients $N_{jk\ell},P_{jk\ell}, Q_{jk\ell}$ depend superquadratically\footnote{A polynomial is \emph{superquadratic} if it contains no constant or linear terms.} on the jet coordinates $v_J$, where $1 \leq \# J \leq j+k+\ell$ in the case of $N_{jk\ell}$,  while $1 \leq \# J \leq j+k+\ell-1$ in the case of $P_{jk\ell}, Q_{jk\ell}$. Moreover, a straightforward inductive argument shows that the coefficient of each monomial appearing therein is a strictly positive integer.
We also note that $P_{jk\ell}$ and $P_{kj\ell}$ are not complex conjugate polynomials. For example, $P_{zzu} = 4 v_z v_{zu} + 2 v_u v_{zz}$ while $P_{\oz\oz u} = 0$.

\section{Preliminary normalizations}\label{sec-NF-general-comp}

In this section we begin the process of constructing a normal form for  nondegenerate hypersurfaces in $\mathbb C^2$ by applying the method of equivariant moving frames.   Using the recurrence relations \eqref{rec-rel} all our computations can be performed symbolically, without requiring coordinate expressions for the prolonged action, the lifted forms \eqref{eq: omega}, or the Maurer--Cartan forms.  In order to efficiently construct a normal form, our principal objective is, at each successive order, to find a cross-section that allows us to normalize as many group parameters or, equivalently, Maurer--Cartan forms {\it as possible}.

Before starting the computations, we observe that the lifted invariants and Maurer--Cartan forms respect the following conjugation relations
\begin{equation*}
\overline{V_J}=V_{\overline J},\qquad
\overline{\mu_J}=\overline{\mu}_{\overline J}, \qquad
\overline{\mcu_J}=\mcu_{\overline J}.
\end{equation*}
All our chosen normalizations will respect these conjugation relations.  Accordingly and for brevity, we will sometimes omit formulae that follow from conjugation.

\subsection{Orders zero and one}

We now start the normalization process by investigating the order zero and order one recurrence relations.  First, from \eqref{rec-rel}, \eqref{varpi}, the order zero recurrence relations are
\begin{subequations}\label{eq: order 01 recurrence relations}
\begin{equation}
dZ = \omega^Z + \mu,\qquad
d\oZ=\omega^{\oZ}+\omu,\qquad
dU = \omega^U + \mcu,\qquad
dV = \varpi + \mcv,
\end{equation}
while, in light of \eqref{eq: order 1 prolonged coefficients}, the order one relations are
\begin{equation}
\begin{aligned}
dV_Z &= \varpi_Z -V_Z\mu_Z - \i V_Z^2\,\mu_U +\i V_ZV_{\oZ}\,\omu_U
-(\i+V_U)\,\mcu_Z + V_Z\, \mcu_U - V_ZV_U\,\mcu_V,\\
dV_{\oZ} &= \varpi_{\oZ} - V_{\oZ}\,\omu_{\oZ} + \i V_{\oZ}^2\, \omu_U
- \i V_ZV_{\oZ}\,\mu_U + (\i - V_U)\,\mcu_{\oZ}+ V_{\oZ}\,\mcu_U - V_{\oZ}V_U\, \mcu_V,\\
dV_U &= \varpi_U - V_Z\,(1+\i V_U)\,\mu_U - V_{\oZ}\,(1-\i V_U)\,\omu_U  -(1+V_U^2)\,\mcu_V.
\end{aligned}
\end{equation}
\end{subequations}
It is therefore possible to normalize
\begin{equation}\label{eq: order 01 normalizations}
Z = \oZ = U = V = V_Z = V_{\oZ} = V_U = 0,
\end{equation}
which effectively normalizes 7 of the pseudo-group parameters\footnote{The non-symbolic moving frame calculus would write out the explicit expressions for these lifted invariants and explicitly solve \eqref{eq: order 01 normalizations} for the pseudo-group parameters, with the (far more complicated) process continuing in a similar fashion throughout.}.
Substituting the normalizations \eqref{eq: order 01 normalizations} into the recurrence relations \eqref{eq: order 01 recurrence relations}, the left-hand sides all vanish while the right-hand sides simplify considerably.  Solving the resulting equations for the partially normalized Maurer--Cartan forms, we obtain
\begin{equation}\label{eq: order 01 partially normalized mc forms}
\mu = -\omega^Z,\quad
\omu = -\omega^{\oZ},\quad
\mcu = -\omega^U,\quad
\mcv = 0,\quad
\mcu_Z = -\i \varpi_Z,\quad
\mcu_{\oZ} = \i\varpi_{\oZ},\quad
\mcu_V = \varpi_U.
\end{equation}
As mentioned at the end of the previous section, in \eqref{eq: order 01 partially normalized mc forms} we omit writing the partial moving frame pull-backs to simplify the notation.

A careful analysis of the prolongation formula \eqref{eq: prolongation formula} yields the following result.

\begin{Lemma}
\label{lem-order-1}
For $k, \ell \geq 0$, we impose the normalizations
\begin{equation}\label{unorm}
V_{Z^{k+1}U^\ell}=V_{\oZ{}^{k+1}U^\ell}=V_{U^{\ell+1}}=0.
\end{equation}
Then, the corresponding phantom recurrence relations
\begin{itemize}
  \item[(a)]  $dV_{Z^{k+1}U^\ell}=dV_{\oZ{}^{k+1}U^\ell}=0$ can be solved for $\mcu_{Z^{k+1}U^\ell},\mcu_{\oZ{}^{k+1}U^\ell}$;
  \item[(b)]  $dV_{U^{\ell+1}}=0$ can be solved for $\mcu_{U^\ell V}$.
\end{itemize}
\end{Lemma}

\proof
Let
\[
\hD_z = \pp{}{z} + v_z \pp{}{v},\qquad
\hD_{\oz} = \pp{}{\oz} + v_{\oz} \pp{}{v},\qquad
\hD_{u} = \pp{}{u} + v_u \pp{}{v}
\]
denote the order zero truncation of the total derivative operators $D_z$, $D_{\oz}$, $D_u$.  In light of the prolongation formula \eqref{eq: prolongation formula}, the terms with the highest order vector field jets in $\phi^J$ are
\[
\bH(\phi^J) = \hD_J\phi - v_z\, \hD_J\xi - v_{\oz}\, \hD_J\oxi - v_u\, \hD_J\eta.
\]
Therefore
\begin{align*}
\bH(\phi^{z^{k+1}u^\ell}) &= \phi_{z^{k+1}u^\ell} + \text{terms involving $v_z$, $v_{\oz}$, $v_u$}
= -\i \eta_{z^{k+1}u^\ell} + \text{terms involving $v_z$, $v_{\oz}$, $v_u$},
\end{align*}
where we used the fact that $\phi_z = -\i \eta_z$.  Thus, the recurrence relation for $V_{Z^{k+1}U^\ell}$, when the normalizations \eqref{eq: order 01 normalizations} have been performed, is
\begin{equation}\label{lemma recurrence relation}
dV_{Z^{k+1}U^\ell} = \varpi_{Z^{k+1}U^\ell} - \i \mcu_{Z^{k+1}U^\ell} + \cdots,
\end{equation}
where the omitted correction terms involve partially normalized Maurer--Cartan forms of order $<k+\ell+1$.  In light of the recurrence relation \eqref{lemma recurrence relation}, it is possible to set $V_{Z^{k+1}U^\ell}=0$ and then solve the resulting phantom recurrence relation for the partially normalized Maurer--Cartan form $\mcu_{Z^{k+1}U^\ell}$.

A similar argument proves part (b), where now
\begin{align*}
\bH(\phi^{u^{\ell+1}}) &= \phi_{u^{\ell+1}} +  \text{terms involving $v_z$, $v_{\oz}$, $v_u$}= -\eta_{u^\ell v} + \text{terms involving $v_z$, $v_{\oz}$, $v_u$},
\end{align*}
thereby completing the proof.
\endproof

Assuming the Maurer--Cartan forms in \eqref{eq: order 01 partially normalized mc forms} and Lemma \ref{lem-order-1} have been normalized, the remaining (partially) unnormalized Maurer--Cartan forms  are
$\mu_{Z^k U^\ell}, \omu_{\oZ{}^k U^\ell}, \mcu_{U^\ell}$, with $k,\ell \in \N_0$ and $k+\ell >0$.

\subsection{Order two}

This order is of rather more interest as it exhibits the role of Levi nondegeneracy when constructing the desired normal form.
Taking into account the normalizations performed in Lemma \ref{lem-order-1}, the only remaining order two recurrence relation is, after simplification,
\begin{equation}\label{eq: VZoZ recurrence relation}
dV_{Z\oZ} = \varpi_{Z\oZ} + V_{Z\oZ}\, (\mcu_U -\mu_Z - \omu_{\oZ}) \equiv V_{Z\oZ}\, (\mcu_U -\mu_Z - \omu_{\oZ}).
\end{equation}
From the group component on the right-hand side, we deduce that $V_{Z\oZ}$ is a {\it relative invariant}, meaning it is mapped to a multiple of itself under the pseudo-group transformations. In particular, the condition $V_{Z\oZ}= 0$ that a hypersurface be degenerate is preserved.  From here on, we leave the degenerate case aside, and exclusively consider nondegenerate hypersurfaces where $V_{Z\oZ}\ne 0$.  We can then normalize
\begin{equation}\label{eq: VZoZ normalization}
V_{Z\oZ}=1.
\end{equation}
Substituting into the recurrence relation \eqref{eq: VZoZ recurrence relation} yields
\begin{equation}\label{alpha_U normalization}
\mcu_U = -\,\varpi_{Z\oZ} + \mu_Z + \omu_{\oZ}.
\end{equation}

Inspecting  the prolonged vector coefficients $\phi^{z\oz u^\ell}$, with $\ell \geq 1$, as given in \eqref{eq: prolongation formula}, we arrive to the following more general result.

\begin{Lemma}
\label{lem-ord-2}
For $\ell\geq 1$,  the phantom recurrence relation for the normalized invariant $V_{Z\oZ U^{\ell}}=0$ can be solved for the (partially) normalized Maurer--Cartan form $\mcu_{U^{\ell+1}}$.
\end{Lemma}

\begin{proof}
We prove the assertion using induction on $\ell$.   For $\ell=1$, we set $V_{Z\oZ U}=0$.  After using the prolongation formula \eqref{eq: prolongation formula} and performing the normalizations in Lemma \ref{lem-order-1} and equation \eqref{eq: VZoZ normalization}, we find that
$$
0 = dV_{Z\oZ U} = \varpi_{Z\oZ U} + \mcu_{UU} - \mu_{ZU}-\mu_{\oZ U}  - V_{Z\oZ U}\,\mu_Z -V_{Z\oZ U}\, \omu_{\oZ} - V_{Z^2\oZ}\,\mu_U - V_{Z\oZ{}^2}\,\omu_U,
$$
which we can solve for the partially normalized Maurer--Cartan form $\mcu_{UU}$.

Now, for $\ell \geq 2$, assume that by setting $V_{Z\oZ U}=0$, $\ldots$, $V_{Z\oZ U^{\ell-1}}=0$, the corresponding recurrence equations can be solved for the partially normalized Maurer--Cartan forms $\mcu_{U^2}, \ldots,\mcu_{U^\ell}$.  By setting $V_{Z\oZ U^\ell} = 0$ our goal is to show that the corresponding recurrence relation can be used to solve for $\mcu_{U^{\ell+1}}$.  Under the correspondence \eqref{eq: substitutions} we must keep track of the vector field jet $\eta_{u^{\ell+1}}$ in the prolonged vector field coefficient $\phi^{z\oz u^\ell}$. First, by the prolongation formula \eqref{eq: prolongation formula} and \eqref{eq: order 1 infinitesimal determining equations}, we have
\begin{align*}
\phi^{z\oz u^\ell}&=D_u^\ell D_z D_{\oz} \big(\phi-v_z\,\xi-v_{\oz}\,\overline\xi-v_u\,\eta\big)+v_{zz\oz u^\ell}\,\xi+v_{z\oz\oz u^\ell}\,\overline\xi+v_{z\oz u^{\ell+1}}\,\eta
\\
&= D_u^\ell \big(\phi^{z\oz}-v_{zz\oz}\,\xi-v_{z\oz\oz}\,\overline\xi-v_{z\oz u}\,\eta\big)+v_{zz\oz u^\ell}\,\xi+v_{z\oz\oz u^\ell}\,\overline\xi+v_{z\oz u^{\ell+1}}\,\eta.
\end{align*}
In light of the infinitesimal determining equations \eqref{eq: infinitesimal determining equations}, the vector field jet $\eta_{u^{\ell+1}}$ will only originate from the derivatives of $\eta$ and $\phi$.  Thus, we set the terms involving  $\xi$ and $\oxi$ and  their derivatives aside and write
\begin{equation}\label{phi-zozul}
\aligned
\phi^{z\oz u^\ell} &=D_u^\ell \big(\phi^{z\oz}-v_{z\oz u}\,\eta\big)+v_{z\oz u^{\ell+1}}\,\eta+\cdots
&= D_u^\ell \phi^{z\oz} - D_u^\ell(v_{z\oz u}\,\eta) + v_{z\oz u^{\ell+1}}\,\eta + \cdots \\
&= D_u^\ell \phi^{z\oz} + \sum_{j=1}^\ell A_j v_{z\oz u^j} + \cdots,
\endaligned
\end{equation}
where the coefficients $A_j$ depend on $v_u$, $\ldots$, $v_{u^\ell}$, and derivatives of $\eta$.  Since by hypothesis $V_{Z\oZ U^j} = 0$, $j=1,\ldots, \ell$, are vanishing phantom invariants, we focus on the first term of \eqref{phi-zozul}.  First,
\[
\phi^{z\oz}=v_{z\oz}\,\big(\eta_u-\overline\xi_{\oz}-\xi_z\big)+\text{terms involving } v_z, v_{\oz}, v_u, v_{zu}, v_{\oz u},
\]
so that
\begin{equation}\label{eq: lemma2 prolongation term}
D_u^\ell(\phi^{z\oz})=v_{z\oz}\,\big(\eta_{u^{\ell+1}}-\overline\xi_{\oz u^\ell}-\xi_{z u^\ell}\big)+\text{terms  involving } \ {v_{zu^l}, v_{\oz u^l}, v_{u^l}, v_{z\oz u^j}},
\end{equation}
where $0\leq l$ and $1\leq j \leq \ell$.  By Lemma \ref{lem-order-1} and the assumptions of the current lemma, since the lifted invariants $V_{ZU^l} = V_{\oZ U^l} = V_{U^l} = V_{Z\oZ U^j} = 0$ are set to zero, the terms omitted in \eqref{eq: lemma2 prolongation term} vanish in the recurrence relation of $V_{Z\oZ U^\ell}$.  Thus we have that
\[
0= dV_{Z\oZ U^\ell} = \varpi_{Z\oZ U^\ell} + \mcu_{U^{\ell+1}} - \omu_{\oZ U^\ell} - \mu_{Z U^\ell} + \cdots,
\]
where the omitted terms involve the Maurer--Cartan forms $\mu_{U^j}$, $\omu_{U^j}$. The last equation can thus be solved for the partially normalized Maurer--Cartan form $\mcu_{U^{\ell+1}}$.
\end{proof}

At this point we have normalized the Maurer--Cartan forms $\mcu_{U^\ell V}$, $\mcu_{Z^k U^\ell}$, $\mcu_{\oZ{}^k U^\ell}$ for all $k, \ell \in \N_0$.  It remains to normalize the Maurer--Cartan forms $\mu_{Z^k U^\ell}, \omu_{\oZ{}^k U^\ell}$, when $k+\ell\geq 1$.

\subsection{Order three}

Taking into account Lemmas \ref{lem-order-1} and \ref{lem-ord-2}, the only remaining non-constant partially normalized third order invariants are $V_{Z^2\oZ},V_{Z\oZ{}^2}$.  After simplification, we find that
\[
dV_{Z^2\oZ}= (V_{Z^3\oZ }-V_{Z^2 \oZ}^2)\,\omega^z + (V_{Z^2\oZ{}^2}-V_{Z\oZ{}^2}V_{Z^2\oZ})\,\omega^{\oz} + V_{Z^2\oZ U}\,\omega^u - \mu_{Z^2} + 4\i \,\omu_U-V_{Z^2\oZ}\,\mu_Z,
\]
and similarly for its conjugate.  It is therefore possible to set $V_{Z^2\oZ}=V_{Z\oZ{}^2}=0$. Solving the corresponding phantom recurrence relations yields the partially normalized Maurer--Cartan forms
\[
\mu_{Z^2}=\varpi_{Z^2\oZ}+4\i\overline\mu_U\qquad \omu_{\oZ{}^2} = \varpi_{Z\oZ{}^2}-4\i\mu_U.
\]
By an argument similar to Lemma \ref{lem-ord-2}, we arrive at the following result.

\begin{Lemma}\label{lem-ord-3}
For $k, \ell\geq 0$, the recurrence relations for the phantom invariants $$V_{Z^{k+2}\oZ U^{\ell}}=V_{Z\oZ{}^{k+2}U^\ell}=0$$
can be solved for the (partially) normalized Maurer--Cartan forms $\mu_{Z^{k+2}U^\ell},\omu_{\oZ{}^{k+2}U^\ell}$.
\end{Lemma}

\subsection{Orders four and five}

In light of the normalizations made in the previous subsections, we now consider the recurrence relations for the remaining non-phantom partially normalized invariants of orders four and five, namely  $V_{Z^2\oZ{}^2},V_{Z^2\oZ{}^3}$, and $V_{Z^3\oZ{}^2}$.  After simplification, we obtain
\begin{align*}
dV_{Z^2\oZ{}^2}&=\varpi_{Z^2\oZ{}^2}-V_{Z^2\oZ{}^2}\,(\mu_Z+\overline\mu_{\oZ})-4\i \,(\mu_{ZU}-\overline\mu_{\oZ U}),
\\
dV_{Z^3\oZ{}^2} &= \varpi_{Z^3\oZ{}^2} - V_{Z^3\oZ{}^2}\,(2\,\mu_Z+\omu_{\oZ}) + 24 \,\omu_{U^2},
\\
dV_{Z^2\oZ{}^3} &= \varpi_{Z^2\oZ{}^3} - V_{Z^2\oZ{}^3}\,(\mu_Z+2\,\omu_{\oZ}) + 24 \,\mu_{U^2}.
\end{align*}
It is therefore possible to set $V_{Z^2\oZ{}^2}=V_{Z^3\oZ{}^2}=V_{Z^2\oZ{}^3}=0$ and use the resulting phantom recurrence relations to solve for the (partially) normalized ${\rm Im} \> \mu_{ZU},\mu_{U^2}, \omu_{U^2}$.  More generally, we have the following result.

\begin{Lemma}\label{lem-ords-4-5}
For $\ell\geq 0$, the recurrence relation for the phantom invariant(s)
\begin{itemize}
  \item[$1)$]  $V_{Z^2\oZ{}^2 U^\ell}=0$ can be solved for the real Maurer--Cartan form ${\rm Im} \> \mu_{ZU^{\ell+1}}$.
  \item[$2)$] $V_{Z^3\oZ{}^2 U^\ell}=V_{Z^2\oZ{}^3U^\ell}=0$ can be solved for the Maurer--Cartan forms $\mu_{U^{\ell+2}},\overline\mu_{U^{\ell+2}}$.
\end{itemize}
\end{Lemma}

In light of the normalizations performed thus far, the remaining unnormalized Maurer--Cartan forms are
$$\mu_U, \quad \overline\mu_U, \quad \mu_Z, \quad \overline\mu_{\oZ}, \quad {\rm Re} \> \mu_{ZU^{\ell+1}},\quad \ell\geq 0.$$

\subsection{Order six}

As demonstrated by Chern and Moser, \cite{Chern-Moser}, this order is of crucial importance. It includes 28 lifted differential invariants but using our previous normalizations, these are all set  to zero except for
$V_{Z^4\oZ{}^2},V_{Z^3\oZ{}^3},V_{Z^2\oZ{}^4}$.  After simplification, we have the recurrence relations
\begin{equation}
\label{eq: rec-order-6}
\aligned
dV_{Z^3\oZ{}^3}&=\varpi_{Z^3\oZ{}^3}+12\,(\mu_{ZU^2}+\overline\mu_{\oZ U^2})-2\,V_{Z^3\oZ{}^3}(\mu_Z+\overline\mu_{\oZ}),
\\
dV_{Z^4\oZ{}^2}&=\varpi_{Z^4\oZ{}^2}-V_{Z^4\oZ{}^2}\,(3\,\mu_Z+\omu_{\oZ}).
\endaligned
\end{equation}
From the first equation it is possible to set $V_{Z^3\oZ{}^3}=0$ and solve for the real part of the partially normalized Maurer--Cartan form $\mu_{ZU^2}$. More generally, we have the following result.

\begin{Lemma}\label{lem-ord-6}
For $\ell\geq 0$, the recurrence relation for the phantom invariant $V_{Z^3\oZ{}^3 U^\ell}=0$ can be solved for the  Maurer--Cartan form ${\rm Re} \>\mu_{ZU^{\ell+2}}$.
\end{Lemma}

Combining part 1) of Lemma \ref{lem-ords-4-5} and Lemma \ref{lem-ord-6}, we have now normalized both $\mu_{ZU^{\ell+2}}$ and $\omu_{\oZ U^{\ell+2}}$ for all $\ell \geq 0$.   Thus, we have successfully normalized all but five of the original infinite collection of Maurer--Cartan forms, namely
\begin{equation}\label{MC-5}
\mu_U, \quad \overline\mu_U, \quad \mu_Z, \quad \overline\mu_{\oZ}, \quad {\rm Re} \> \mu_{ZU}.
\end{equation}
From this result, we can already deduce that the isotropy group of a nondegenerate hypersurface is at most five-dimensional (cf. \cite{Valiquette-SIGMA}).

Combining all the normalizations made thus far, we are able to holomorphically transform a local neighborhood of the nondegenerate hypersurface $M$ to the partial normal form
\begin{equation}
\label{NF-pre}
v=z\oz+\frac{1}{4!\, 2!} \big(V_{Z^4\oZ{}^2} \, z^4 \oz^2+V_{Z^2\oZ{}^4}\, z^2 \oz^4\big)+\sum_{\scriptstyle j+k+\ell\geq 7 \atop {\scriptstyle j\geq 2, \ k \geq 4, \ \ell \geq 0,\ostrut84\atop \scriptstyle {\rm or } \ j\geq 4, \ k \geq 2, \ \ell \geq 0.}} \frac{1}{j!\, k! \,\ell!} V_{Z^j \oZ{}^k U^\ell} z^j \oz^k u^\ell,
\end{equation}
which coincides with the normal form \eqref{NF-CM-general} found by Chern and Moser.   See their paper \cite[Theorem 3.5]{Chern-Moser} for a proof of the convergence of the partial normal form expansion, based on the concept of chains on hypersurfaces.
The second equation in \eqref{eq: rec-order-6} implies that the partially normalized function
\begin{equation*}
V_{Z^4\oZ{}^2} = 48\, c_{42}(0) = 6\,\CMR
\end{equation*}
is a relative invariant.  The multiple $\CMR = V_{Z^4\oZ{}^2}/6$ is known, \cite[Lemma 8.7]{Jacobowitz}, \cite[Theorem 5.2]{Sabzevari-Merker-14}, as the {\it Cartan curvature} of the nondegenerate hypersurface $M$ at the point $\bp$.  

\begin{Definition}
\label{umbilic}
If the Cartan curvature vanishes, so $V_{Z^4\oZ{}^2}|_{\bp} =0$, then $\bp \in M$ is called an \emph{umbilic point}.
\end{Definition}

In other words, a point $\bp$ is umbilic if and only if the corresponding normal form contains no sixth order terms.  As in Euclidean surface theory, one must modify the  subsequent analysis at an umbilic point, and this depends upon whether or not the hypersurface is umbilic in a neighborhood of the point.

Before continuing, let us analyze the normalizations performed in Lemmas \ref{lem-order-1}--\ref{lem-ord-6} in more details, arriving at the following result.

\begin{Proposition}\label{prop-gen-patt}
The following (partially) normalized Maurer--Cartan forms
$$\mu \equiv \omu \equiv \mcu \equiv \mcv \equiv \mcu _Z \equiv \mcu _{\oZ} \equiv \mcu _V \equiv 0$$
are linear combinations of the horizontal forms $\omega^Z$, $\omega^{\oZ}$, $\omega^U$.   Furthermore, for $k,\ell \geq 0$, the higher order (partially) normalized Maurer--Cartan forms are given, modulo the horizontal forms, by
\begin{enumerate}[label=(\alph*)]
\setlength{\itemsep}{0.25cm}

\item \label{alpha-U-k-V}
$\mcu_{U^\ell V}\equiv 0;$

\item\label{Im-z-u-l}
${\rm Im} \> \mu_{ZU^{\ell+1}}\equiv 0;$

\item \label{alpha-u-l+1-general}
$\displaystyle \mcu_{U} \equiv \mu_Z + \omu_{\oZ} , \qquad
\mcu_{U^2} \equiv 2\, \mu_{ZU}, \qquad
\mcu_{U^3} \equiv 2 \,\mu_{ZU^2} \equiv 0$, \hfill\break
$\displaystyle \mcu_{U^{\ell+4}} \equiv 2\, \mu_{ZU^{\ell+3}} \equiv \frac{1}{12} \sum_{j=1}^{\ell+1} \binom{\ell+1}{j} \Big[ V_{Z^4\oZ{}^3 U^{\ell+1-j}}\,\mu_{U^j}+ V_{Z^3\oZ{}^4 U^{\ell+1-j}}\,\overline\mu_{U^j}\Big];$

\item \label{Re-mu-z-u-l+2}
$\displaystyle \mu_{ZU^{2}} = \omu_{\oZ U^{2}}  \equiv 0 , $ \hfill\break $\displaystyle \mu_{ZU^{\ell+3}} = \omu_{\oZ U^{\ell+3}} \equiv \frac{1}{24}  \sum_{j=1}^{\ell+1} \binom{\ell+1}{j} \Big[ V_{Z^4\oZ{}^3 U^{\ell+1-j}}\,\mu_{U^j}+ V_{Z^3\oZ{}^4 U^{\ell+1-j}}\,\overline\mu_{U^j}\Big];$

\item \label{alpha-z-j-u-l-general}
$\displaystyle \mcu_{ZU^\ell} \equiv \i \overline\mu_{U^\ell} , \qquad
\mcu_{\oZ U^\ell} \equiv -\i \mu_{U^\ell} , \qquad
\mcu_{Z^{k+2} U^\ell} \equiv \mcu_{\oZ{}^{k+2} U^\ell} \equiv 0;$

\item \label{mu-z-j+2-u-l-general}
$\mu_{Z^{2} U^\ell} \equiv 4\i \overline\mu_{U^{\ell+1}}$, \qquad  $\mu_{Z^{3} U^\ell} \equiv 0$, \qquad 
$\mu_{Z^{k+4}}\equiv 0$, \\ 
$\displaystyle \mu_{Z^{k+4} U^{\ell+1}}\equiv -\sum_{j=1}^{\ell+1} \binom{\ell+1}{j} \,V_{Z^{k+4}\oZ{}^2 U^{\ell+1-j}}\,\overline\mu_{U^j}$, \hfill\break
$\overline\mu_{\oZ{}^{2} U^\ell} \equiv -4\i \mu_{U^{\ell+1}}$, \qquad 
$\overline\mu_{\oZ{}^{3} U^\ell} \equiv 0$, \qquad   
$\overline\mu_{\oZ{}^{k+4}} \equiv 0$, \\  
$\displaystyle \overline\mu_{\oZ{}^{k+4} U^{\ell+1}} \equiv  -\sum_{j=1}^{\ell+1} \binom{\ell+1}{j}\,V_{Z^2\oZ{}^{k+4} U^{\ell+1-j}}\mu_{U^j};$

\item \label{mu-u-l+2}
$\mu_{U^2} \equiv \overline\mu_{U^2} \equiv 0$, while $\mu_{U^{\ell+3}},\overline\mu_{U^{\ell+3}},$ are expressed in terms of $\mu_U$ and $\overline\mu_{U}$ via the recursive relations
\[
\mu_{U^{\ell+3}} \equiv \frac{1}{24} \sum_{j=1}^{\ell+1} \binom{\ell+1}{j}
 V_{Z^2\oZ{}^4 U^{\ell+1-j}}\overline\mu_{U^j},\qquad
\overline\mu_{U^{\ell+3}} \equiv \frac{1}{24} \sum_{j=1}^{\ell+1} \binom{\ell+1}{j} V_{Z^4\oZ{}^2 U^{\ell+1-j}}\mu_{U^j}.
\]
\end{enumerate}
\end{Proposition}

\noindent {\it Remark\/}: We have listed all the complex conjugate relations for the reader's convenience, but will only give the details of the proof of one of each pair below.

\begin{proof}
Let
\begin{equation}\label{cross-section}
\begin{aligned}
\mathcal{K} = \{v_{z\oz} = 1,\; z&=\oz=u=v=v_{z^{k}u^\ell}=v_{\oz^{k}u^\ell}=v_{z\oz u^{\ell+1}} =v_{z^{k+2}\oz u^\ell}\\
&= v_{z \oz^{k+2}u^\ell}=v_{z^2\oz^2u^\ell}=v_{z^3\oz^2u^\ell}=v_{z^2\oz^3u^\ell}=v_{z^3\oz^3u^\ell}=0\;|\; k,\, \ell \in \N_0\}
\end{aligned}
\end{equation}
be the cross-section corresponding to the normalizations performed in Lemmas \ref{lem-order-1}--\ref{lem-ord-6}; see also the normal form \eqref{NF-pre}.  Moreover, from the order zero normalizations \eqref{eq: order 01 partially normalized mc forms}, invariantization of the last three terms in the prolongation formula \eqref{eq: prolongation formula} produces a horizontal form
\[
v_{J,z}\,\xi  + v_{J,\oz} \,\oxi + v_{J,u}\, \eta \quad \overset{\eqref{eq: substitutions}}{\longmapsto }\quad V_{J,Z} \,\mu  + V_{J,\oZ}\,\omu + V_{J,U}\,\mcu \equiv 0,
\]
and hence we introduce the equality
\begin{equation}\label{phiJ}
\phi^J \equiv D_J(\phi - v_z\,\xi  - v_{\oz}\,\oxi  - v_u\,\eta )\qquad  \bmod \ (\xi,\oxi,\eta).
\end{equation}
We note that, if $F(z,\oz,u,v)$ is any function of the base variables, then, on the cross-section \eqref{cross-section}, the following total derivatives reduce to the corresponding partial derivatives
\begin{equation}\label{Ds}
D_u^\ell F\big|_{\mK} = \partial_u^\ell F, \qquad D_z^k D_u^\ell F\big|_{\mK} = \partial_z^k\partial_u^\ell F, \qquad D_{\oz}^k D_u^\ell F\big|_{\mK} = \partial_{\oz}^k\partial_u^\ell F.
\end{equation}

\begin{description}
\item[Proof of part \ref{alpha-U-k-V}] According to \eqref{phiJ},
\[
\phi^{u^{\ell+1}} \equiv D_u^{\ell+1}\big(\phi - v_z\,\xi  - v_{\oz}\,\oxi  - v_u\,\eta\big).
\]
In view of \eqref{Ds} and the determining equations \eqref{eq: infinitesimal determining equations}, when restricted to the cross-section
\[
\phi^{u^{\ell+1}}\big|_{\mK} \equiv \phi_{u^{\ell+1}} = -\eta_{u^\ell v}.
\]
We thus have the recurrence relation
\[
0 = dV_{U^{\ell+1}} \equiv  -\mcu_{U^\ell V}.
\]

\item[Proof of part \ref{Im-z-u-l}]
According to \eqref{phiJ},
\begin{equation*}
\phi^{z^2\oz{}^2u^\ell} \equiv D_u^\ell \big(\phi^{z^2\oz{}^2}-v_{z^3\oz{}^2}\,\xi-v_{z^2\oz{}^3}\,\overline\xi- v_{z^2\oz{}^2u}\,\eta\big).
\end{equation*}
On the cross-section \eqref{cross-section},
\[
D_u^\ell\big(v_{z^3\oz^2}\,\xi +v_{z^2\oz^3}\,\overline\xi +v_{z^2\oz^2u}\,\eta \big)\big|_{\mK}=0.
\]
Furthermore, the determining equations \eqref{eq: infinitesimal determining equations} imply
\[
\phi^{z^2\oz^2}\big|_{\mK}
=8\,{\rm Im} \>\xi_{zu}+2\,\eta_{uv},
\]
and hence
\[
\phi^{z^2\oz^2 u^\ell}\big|_{\mK} \equiv D_u^\ell(\phi^{z^2\oz^2})\big|_{\mK} = 8\,{\rm Im} \> \xi_{zu^{\ell+1}}+2\,\eta_{u^{\ell+1} v}.
\]
From this, we deduce the phantom recurrence relation
\[
0=dV_{Z^2\oZ{}^2 U^\ell}\equiv 8\,{\rm Im} \> \mu_{ZU^{\ell+1}}+2\,\mcu_{U^{\ell+1} V}.
\]
By part \ref{alpha-U-k-V}, $\mcu_{U^{\ell+1} V}\equiv 0$, which implies that ${\rm Im} \> \mu_{ZU^{\ell+1}} \equiv 0$.

\bigskip

\item[Proof of parts   \ref{alpha-u-l+1-general} and \ref{Re-mu-z-u-l+2}]
The formula for $\mcu _U$ follows from \eqref{alpha_U normalization}. As for the formulae for $\mcu_{U^{\ell}}$, when $\ell \geq 2$, we first note that
\[
\aligned
\phi^{z\oz u^\ell} \equiv D_u^\ell \big(\phi^{z\oz}-v_{z^2\oz}\,\xi -v_{z\oz{}^2}\,\overline\xi -v_{z\oz u}\,\eta \big).
\endaligned
\]
Since
\[
D_u^\ell \big(v_{z^2\oz}\,\xi +v_{z\oz{}^2}\,\overline\xi +v_{z\oz u}\,\eta \big)\big|_{\mK} = 0,
\]
by \eqref{eq: lemma2 prolongation term},
\[
\phi^{z\oz u^\ell}\big|_{\mK} \equiv D_u^\ell(\phi^{z\oz})\big|_{\mK} =  \eta_{u^{\ell+1}}-\overline\xi_{\oz u^\ell}-\xi_{z u^\ell}.
\]
We deduce the recurrence relation
\[
0=dV_{Z\oZ U^\ell} \equiv  \mcu_{U^{\ell+1}}-\mu_{ZU^\ell}-\overline\mu_{\oZ U^\ell},
\]
which implies
\[
\mcu_{U^{\ell+1}} \equiv  \mu_{ZU^\ell}+\overline\mu_{\oZ U^\ell} .
\]
Moreover, when $\ell\geq 1$, part \ref{Im-z-u-l} implies that $\mu_{ZU^\ell}$ is a real form, which establishes the first formulae for $\mcu_{U^{\ell+1}}$.
Next, from the prolongation formula \eqref{eq: prolongation formula} and \eqref{phiJ},
\begin{equation}\label{phizzz}
\phi^{z^3\oz^3 u^\ell} \equiv D_u^\ell \big(\phi^{z^3\oz^3}-v_{z^4\oz^3}\,\xi -v_{z^3\oz^4}\,\overline\xi -v_{z^3\oz^3u}\,\eta \big).
\end{equation}
On the cross-section
$$\qquad \phi^{z^3\oz^3}\big|_{\mK} = 36\,{\rm Re} \>(\xi_{zuu}) - 6 \,\eta _{uuu}, \quad \hbox{and hence} \quad
D_u^\ell(\phi^{z^3\oz^3})\big|_{\mK}
= 36\, {\rm Re} \>(\xi_{zu^{\ell+2}}) - 6 \,\eta _{u^{\ell+3}}.
$$
Next, applying  $D_u^\ell$ to $v_{z^4\oz^3}\xi+v_{z^3\oz^4}\overline\xi$, the non-vanishing terms on the cross-section are
$$D_u^\ell( v_{z^4\oz^3}\,\xi+v_{z^3\oz^4}\,\overline\xi )\big|_{\mK} = \sum_{j=0}^\ell \binom{\ell}{j}\big[v_{z^4\oz^3 u^{\ell-j}}\,\xi_{u^j} +v_{z^3\oz^4 u^{\ell-j}}\,\overline\xi_{u^j} \big].
$$
Finally, we note that
\[
D_u^\ell(v_{z^3\oz^3u}\,\eta)\big|_{\mK} = 0.
\]
Substituting these equations into \eqref{phizzz}, we conclude that, modulo the horizontal coframe,
\begin{equation}\label{Re-mu-z-u-l+2 recurrence relation}
\begin{aligned}
0=dV_{Z^3\oZ{}^3 U^\ell} &\equiv 36\,{\rm Re} \> \mu_{ZU^{\ell+2}} - 6 \,\mcu _{U^{\ell+3}}-\sum_{j=0}^\ell\binom{\ell}{j}
\Big[V_{Z^4\oZ{}^3 U^{\ell-j}}\,\mu_{U^j} +V_{Z^3\oZ{}^4 U^{\ell-j}}\,\overline\mu_{U^j} \Big]
\\
&\equiv 24\,{\rm Re} \> \mu_{ZU^{\ell+2}} -\sum_{j=0}^\ell\binom{\ell}{j}
\Big[V_{Z^4\oZ{}^3 U^{\ell-j}}\,\mu_{U^j} +V_{Z^3\oZ{}^4 U^{\ell-j}}\,\overline\mu_{U^j} \Big],
\end{aligned}
\end{equation}
where we used the initial formulae in part   \ref{alpha-u-l+1-general} in the second line.
Part \ref{Im-z-u-l} implies that $\mu_{ZU^{\ell+2}} = {\rm Re} \> \mu_{ZU^{\ell+2}}$ is a real form, and hence \eqref{Re-mu-z-u-l+2 recurrence relation} yields part \ref{Re-mu-z-u-l+2}, which then completes the proof of part   \ref{alpha-u-l+1-general}.

\bigskip

\item[Proof of part  \ref{alpha-z-j-u-l-general}]
As in Lemma \ref{lem-order-1}, the formula for $\mcu_{ZU^\ell}$ follows from the recurrence relation $0 = d V_{ZU^\ell}$. Also, by \eqref{phiJ}, if $k \geq 2$,
\[
\phi^{z^k u^\ell}\big|_{\mK} \equiv D_z^k D_u^\ell \big(\phi-v_z\,\xi -v_{\oz}\,\overline\xi -v_u\,\eta \big)\big|_{\mK}\equiv \phi_{z^k u^\ell} = -\i \eta _{z^k u^\ell},
\]
which yields
$0=dV_{Z^k U^\ell} \equiv  -\i \mcu_{Z^k U^\ell}$ when $k \geq 2$,
and the result follows.

\bigskip

\item[Proof of part \ref{mu-z-j+2-u-l-general}]
Again, by \eqref{phiJ},
\begin{equation}\label{prolonged coefficient mu-z-j+2-u-l-general}
\phi^{z^{k+2} \oz u^\ell} \equiv D_z^{k+2} D_{\oz} D_u^\ell \big(\phi-v_z\,\xi -v_{\oz}\,\overline\xi -v_u\,\eta \big).
\end{equation}
First note that
\[
D_z^{k+2} D_{\oz} D_u^\ell(v_u\,\eta )\big|_{\mK} = 0.
\]
Next,
\[
D_z^{k+2} D_{\oz} D_u^\ell (\phi)\big|_{\mK} = (k+2)\, v_{z\oz}\,\phi_{z^{k+1} u^\ell v} \big|_{\mK} = (k+2)\, \eta_{z^{k+1} u^{\ell+1}}.
\]
Finally,
\[
\aligned
D_z D_{\oz}\big(v_{z}\,\xi +v_{\oz}\,\overline\xi\big)=&\big(v_z v_{\oz}\,\xi_{zv} +v_z^2 v_{\oz}\,\xi_{vv}+\underline{v_{z\oz}\,\xi_{z} } +v_{\oz} v_{zz}\, \xi_{v}+2\,v_z v_{z\oz}\,\xi_{v}  +v_{zz\oz}\,\xi
\\
&+v_z v_{\oz}\,\overline\xi_{\oz v}+v_z v_{\oz}^2\,\overline\xi_{vv} +v_{z\oz}\,\overline\xi_{\oz}  +\underline{2\,v_{\oz}v_{z\oz}\, \overline\xi_{v} }+v_z v_{\oz\oz}\,\overline\xi_{v} +\underline{v_{z\oz\oz}\,\overline\xi }\big).
\endaligned
\]
After applying $D_z^{k+1} D_u^\ell$ to the last expression, one observes that on the cross-section the only non-zero terms emerge from the three underlined terms.  First, a careful inspection yields
\[
D_z^{k+1} D_u^\ell \big(v_{z\oz}\,\xi_z +v_{z\oz^2}\,\overline\xi\big)\big|_{\mK}
=
\xi_{z^{k+2} u^\ell} +\sum_{j=0}^\ell \binom{\ell}{j}v_{z^{k+2}\oz{}^2 u^{\ell-j}}\,\overline\xi_{u^j}.
\]
In particular, when $k=0$ or $1$, the sum vanishes on $\mathcal{K}$.  Moreover,
\[
D_z^{k+1} D_u^\ell \big(2\,v_{z\oz} v_{\oz}\,\overline\xi_v \big)|_{\mK} =
\begin{cases}
-2\i \,\overline\xi_{u^{\ell+1}}, & k=0,\\
\hspace{0.75cm}0,&k\geq 1.
\end{cases}
\]
In summary, the part of \eqref{prolonged coefficient mu-z-j+2-u-l-general} that does not vanish on the cross-section is
\[
\begin{cases}
\displaystyle 2\,\eta_{z u^{\ell+1}} +2\i\,\overline\xi_{u^{\ell+1}} -\xi_{z^{2} u^\ell},  & k=0,\\
\displaystyle (k+2)\,\eta_{z^{k+1} u^{\ell+1}} -\xi_{z^{k+2} u^\ell} -\sum_{j=0}^\ell \binom{\ell}{j}v_{z^{k+2}\oz{}^2 u^{\ell-j}}\,\overline\xi_{u^j} ,&  k\geq 1.
\end{cases}
\]
Finally, using the result in part \ref{alpha-z-j-u-l-general} we obtain the recurrence relations
\[
0=dV_{Z^{2+k}\oZ U^\ell} \equiv \begin{cases}
\displaystyle 4\i \overline\mu_{U^{\ell+1}}-\mu_{Z^2 U^\ell}, &   k=0,\\
\displaystyle -\mu_{Z^{k+2} U^\ell}-\sum_{j=0}^\ell \binom{\ell}{j}
V_{Z^{k+2}\oZ{}^2 U^{\ell-j}}\,\overline\mu_{U^j},  &    k\geq 1.
\end{cases}
\]
Solving for $\mu_{Z^{k+2}U^\ell}$ yields the result.

\bigskip

\item[Proof of part \ref{mu-u-l+2}]
In this case, we use
\[
\phi^{z^2\oz^3u^\ell} \equiv D_u^\ell \big(\phi^{z^2\oz^3}-v_{z^3\oz^3}\,\xi -v_{z^2\oz^4}\,\overline\xi -v_{z^2\oz^3u}\,\eta \big).
\]
Since
\[
\phi^{z^2\oz^3}\big|_{\mK} = 6\,\xi_{u^2}+6\i\,\overline\xi_{\oz^2u}-6\i\,\eta_{\oz u^2},
\]
we find that
\begin{equation*}
D_u^\ell(\phi^{z^2\oz^3})\big|_{\mK} = 6\,\xi_{u^{\ell+2}}+6\i\,\overline\xi_{\oz^2u^{\ell+1}}-6\i\,\eta_{\oz u^{\ell+2}}.
\end{equation*}
Also,
\[
-D_u^\ell(v_{z^3\oz^3}\,\xi +v_{z^2\oz^3u}\,\eta )\big|_{\mK} = 0.
\]
Finally,
\begin{equation*}
D_u^\ell\big(v_{z^2\oz^4}\,\overline\xi )\big|_{\mK} = \sum_{j=0}^\ell \binom{\ell}{j}v_{z^2\oz^4 u^{\ell-j}}\,\overline\xi_{u^j} .
\end{equation*}
This yields the recurrence relations
\[
\aligned
0=dV_{Z^2\oZ{}^3 U^\ell}&\equiv 6\,\mu_{U^{\ell+2}}+6\i\overline\mu_{\oZ{}^2U^{\ell+1}}-6\i \mcu_{\oZ U^{\ell+2}} -\sum_{j=0}^\ell \binom{\ell}{j}
V_{Z^2\oZ{}^4 U^{\ell-j}}\,\overline\mu_{U^j}
\\
&\equiv 24\,\mu_{U^{\ell+2}} -\sum_{j=0}^\ell \binom{\ell}{j} V_{Z^2\oZ{}^4 U^{\ell-j}}\,\overline\mu_{U^j} ,
\endaligned
\]
where we used parts \ref{alpha-z-j-u-l-general}, \ref{mu-z-j+2-u-l-general} in the second line. Solving for $\mu_{U^{\ell+2}}$ completes the proof.
\end{description}
\end{proof}

\begin{Remark}
\label{rem}
We note that, in Proposition \ref{prop-gen-patt}, all Maurer--Cartan forms  except for $\mcu_U=\mu_Z+\overline\mu_{\oZ}$ and $\mcu_{UU}=2\, {\rm Re} \>\mu_{ZU}$ are, modulo horizontal forms, either zero or expressed as a linear combination of only $\mu_U$ and $\overline\mu_U$ with differential invariant coefficients.
\end{Remark}

\section{Normal forms for Nondegenerate Hypersurfaces in $\C^2$}

In this section we complete our normal form constructions.
According to Definition \ref{umbilic}, the Cartan curvature  $\CMR = V_{Z^4\oZ{}^2}/6$ distinguishes between non-umbilic and umbilic points on the hypersurface.  The latter subdivide into the locally umbilic case, where $\CMR$ vanishes on a neighborhood of $\bp$, and the singularly umbilic points.  As we will see, the singularly umbilic case further splits into three subclasses that must be treated differently: generic, semi-circular,  and  circular, the latter being distinguished by the presence of a one-parameter automorphism group fixing $\bp$.  The precise definitions of these three subclasses can be found in Definition \ref{singularly-umbilic-def} below.

The first two classes, namely non-umbilic and locally umbilic hypersurfaces, are well known through the original work of Chern and Moser, \cite{Chern-Moser}; see also \cite{Jacobowitz, Webster-78}. Here we rederive these classical results using the equivariant moving frame calculus, which will allow the reader to compare the two  approaches.  We then show, in the former case, how the moving frame recurrence formulae provide the tools that enable one to investigate the structure of the differential invariant algebra in detail; see Theorem \ref{J} below.

Recall that, by virtue of the normalizations in Section \ref{sec-NF-general-comp} used to achieve the partial normal form \eqref{NF-pre}, the remaining non-phantom partially normalized differential invariants are $V_{Z^j \oZ{}^k U^\ell}$ for $j  \geq 4, k \geq 2,\ell \geq 0$, and their conjugates $V_{Z^k \oZ{}^j U^\ell}$.  There are five remaining unnormalized Maurer--Cartan forms, listed in \eqref{MC-5}, which correspond to a residual five-dimensional equivalence group that preserves the aforementioned partial normal form.

\subsection{Non-umbilic hypersurfaces}\label{non-umb}

Suppose first that the Cartan curvature does not vanish and so the point $\bp \in M$ is not umbilic.  Accordingly, we continue the computations of Section \ref{sec-NF-general-comp} under the assumption that the relative invariant  $V_{Z^4\oZ{}^2}$ does not vanish at $\bp$, which implies it and its conjugate are nonzero in a neighborhood thereof. In this case, we can normalize\footnote{The choice of $48$ as the normalization constant is so that we can precisely reproduce the normal form \eqref{NF-CM-non-umb} of Chern and Moser.  Any other nonzero constant would work equally well, with corresponding modifications to the coefficients of the resulting formulae.}
\begin{equation}
V_{Z^4\oZ{}^2}=V_{Z^2\oZ{}^4}=48,
\end{equation}
which corresponds to setting $c_{42}(0) =1$ or, equivalently, the
Cartan curvature $\CMR = 8$.
 The second recurrence relation in \eqref{eq: rec-order-6} then implies
\begin{equation*}
\mu_Z = \frac{1}{384}\,(3\,\varpi_{Z^4 \oZ{}^2}-\varpi_{Z^2\oZ{}^4})
\qquad \text{and}\qquad \omu_{\oZ} = \frac{1}{384}\,(3\,\varpi_{Z^2\oZ{}^4}-\varpi_{Z^4\oZ{}^2}).
\end{equation*}
To normalize the remaining Maurer--Cartan forms $\mu_U, \overline\mu_U, {\rm Re} \> \mu_{ZU}$, we proceed to order 7, where there are 36 lifted invariants, 30 of which have already been normalized to zero, leaving
\[
V_{Z^5\oZ{}^2},\quad V_{Z^4\oZ{}^3},\quad V_{Z^4\oZ{}^2U},\quad
V_{Z^2\oZ{}^4U},\quad V_{Z^3\oZ{}^4},\quad V_{Z^2\oZ{}^5}.
\]
We now consider the following recurrence relations
\begin{equation*}
\aligned
dV_{Z^4\oZ{}^2U} &= \varpi_{Z^4\oZ{}^2U} -\frac{1}{192} V_{Z^4\oZ{}^2U}\,(\varpi_{Z^2\oZ{}^4} + 5\, \varpi_{Z^4\oZ{}^2})-(V_{Z^5\oZ{}^2}\,\mu_U+V_{Z^4\oZ{}^3}\,\omu_U)-192\,{\rm Re} \>\mu_{ZU},
\\
dV_{Z^4\oZ{}^3} &= \varpi_{Z^4\oZ{}^3} -\frac{1}{384}V_{Z^4\oZ{}^3}\,(3\,\varpi_{Z^2\oZ{}^4}+7\,\varpi_{Z^4\oZ{}^2})-7\i V_{Z^4\oZ{}^2U}\,\omega^Z -96\i \mu_U.
\endaligned
\end{equation*}
By normalizing
\begin{equation}
{\rm Re} \>(V_{Z^4\oZ{}^2U}) = V_{Z^4\oZ{}^3}=V_{Z^3\oZ{}^4}=0,
\end{equation}
and using \eqref{varpi}, one can solve the resulting phantom recurrence relations for the remaining Maurer--Cartan forms
\begin{equation}
\aligned
\mu_U &= -\frac{\i}{96} \varpi_{Z^4\oZ{}^3} - \frac 7{96} V_{Z^4\oZ{}^2U}\,\omega^Z,
\qquad \overline\mu_U = \frac{\i}{96} \varpi_{Z^3\oZ{}^4} - \frac 7{96} V_{Z^2\oZ{}^4U}\,\omega^{\oZ},\\
{\rm Re} \> \mu_{ZU} 
 &= \frac{1}{384} \big(\varpi_{Z^4\oZ{}^2U} + \varpi_{Z^2\oZ{}^4U}\big) + \frac{\i}{36864}\big(V_{Z^5\oZ{}^2}\, \varpi_{Z^4\oZ{}^3} - V_{Z^2\oZ{}^5}\,\varpi_{Z^3\oZ{}^4}\big) \\
&\hskip1cm + \frac{5\i}{36864} {\rm Im} \>(V_{Z^4\oZ{}^2U})\,\big(V_{Z^5\oZ{}^2}\,\omega^Z - V_{Z^2\oZ{}^5}\omega ^{\oZ}\big) + \frac{1}{9216} \bigl[{\rm Im} \>(V_{Z^4\oZ{}^2U})\bigr]^2\,\omega^U.
\endaligned
\end{equation}
At this point, since we have normalized all the Maurer--Cartan forms, the prolonged action has become locally free, and hence there is at most a discrete isotropy subgroup remaining. According to \cite[Theorem 3]{Beloshapka-80}, this subgroup is either trivial or has exactly two elements. We have thereby re-established the classical Chern--Moser normal form \eqref{NF-CM-non-umb} for the hypersurface at a non-umbilic point. Convergence of the normal form expansion is assured by Chern and Moser's proof of convergence of the partial normal form \eqref{NF-pre} for any values of the remaining pseudo-group parameters, including those specified by the moving frame.

\begin{Theorem}\label{nonumbilic}
Let $M\subset \mathbb{C}^2$ be a nondegenerate hypersurface and  $\bp\in M$ a non-umbilic point. Then, depending on the cardinality of the discrete isotropy subgroup at $\bp$, there exists exactly one or two holomorphic transformations mapping it to the convergent normal form
\begin{equation}\label{nonumbilicnf}
\aligned
v=z\oz &+(z^4\oz^2+z^2\oz^4)+\frac{1}{240}V_{Z^5\oZ{}^2} z^5\oz^2+\frac{1}{240}V_{Z^2\oZ{}^5} z^2\oz^5 \\
&+\frac{1}{48} V_{Z^4\oZ{}^2U}(z^4\oz^2u-z^2\oz^4 u)
 +\sum_{\scriptstyle j+k+\ell \geq 8
\atop{\scriptstyle j\geq 2, \ k \geq 4, \ \ell \geq 0,\ostrut84\atop \scriptstyle {\rm or } \ j\geq 4, \ k \geq 2, \ \ell \geq 0}} \frac{1}{j!\, k! \,\ell!} V_{Z^j \oZ{}^k U^\ell} z^j \oz^k u^\ell,
\endaligned
\end{equation}
where $V_{Z^4\oZ{}^2U}$ is purely imaginary, while $V_{Z^2\oZ{}^5} = \overline{V_{Z^5\oZ{}^2}}$.  Furthermore, modulo certain changes in sign,
the non-constant real and imaginary parts of the coefficients $V_J$ provide a complete set of functionally independent differential invariants in a neighborhood of $\bp$.
\end{Theorem}

Comparing with \eqref{NF-CM-non-umb}, we deduce that the Chern--Moser invariants are given by
\[
\CMJ = \frac{1}{240}V_{Z^5\oZ{}^2},\qquad \CMoJ = \frac{1}{240}V_{Z^2\oZ{}^5},\qquad \CMK = -\frac{\i}{48} V_{Z^4\oZ{}^2U}= \frac{\i}{48} V_{Z^2\oZ{}^4U}.
\]
As noted by Chern and Moser, \cite{Chern-Moser}, the non-umbilic normal form \eqref{nonumbilicnf} is unique modulo the map $z \mapsto -z$, which changes the sign of all the terms that are of odd order in $z,\oz$, so that
\begin{equation} \label{jksign}
V_{Z^j\oZ{}^kU^\ell} \ \longmapsto \ (-1)^{j+k} \,V_{Z^j\oZ{}^kU^\ell}.
\end{equation}
This implies that the prolonged action of $\G$ is free on the \emph{fully regular subset}
\begin{equation} \label{fully regular}
\mathcal V\n = \mathcal U\n \setminus \mathcal S\n \quad \hbox{for} \quad n \geq 7,
\end{equation}
where $\mathcal U\n = \Jn \setminus \{\CMR \ne 0\}$ denotes the set of non-umbilic jets of order $n \geq 6$, while
\begin{equation*}
\mathcal S\n = \bigl\{V_{Z^j\oZ{}^kU^\ell} = 0 \quad \hbox{for} \quad 7 \leq j+k+\ell \leq  n, \ \ j + k = 2m + 1, \ \ m \in \N\bigr\}
\end{equation*}
is the singular subset where the prolonged action in not free.
We will call $\bp \in M$ a \emph{fully regular point of order $n \geq 7$} if the jet of $M$ at $\bp$ lies in $\mathcal V\n$, and $M$ \emph{fully regular of order $n$} if all its points are of this form.

The resulting non-phantom recurrence formulae for the differential invariants all have the form\footnote{As always, we are ignoring contact forms and only writing out the horizontal components of the differential.}
\begin{equation}\label{dVJ}
dV_J = (\DZ V_J)\,\omega^Z + (\DoZ V_J)\,\omega^{\oZ} + (\DU V_J)\,\omega^U.
\end{equation}
Here $\DZ,\DoZ,\DU$ are the invariant differential operators dual to the invariant horizontal forms $\omega^Z,\omega^{\oZ},\omega^U$.  As such, they map differential invariants to higher order differential invariants, and can thus be applied repeatedly to generate an infinite hierarchy of differentiated invariants. On the other hand, using the recurrence relations \eqref{rec-rel} and the normalization formulae, we can rewrite the left-hand side as a linear combination of the invariant horizontal forms whose coefficients are certain polynomial combinations of the normalized differential invariants $V_J$. Thus, the recurrence relations provide expressions for the differentiated invariants in terms of the basic invariants, and hence determines the structure of the associated differential invariant algebra.  It is worth re-emphasizing that this can all be done purely symbolically, and requires no knowledge of the explicit formulae for the differential invariants or the invariant differential operators or even the moving frame!

A set of differential invariants is called \emph{generating} if one can obtain all the other differential invariants as algebraic combinations of the repeated invariant derivatives of those in the generating set.  The general Lie--Tresse Theorem, \cite{Kruglikov-Lychagin-16,Olver-Pohjanpelto-09} states that a Lie pseudo-group that acts (locally) freely on an open subset of a sufficiently high order jet space --- including the holomorphic pseudo-group $\G$ under consideration --- possesses a finite system of generating differential invariants.  In our case, since the pseudo-group acts locally freely at order $7$, according to a general result from the method of moving frames, \cite{Olver-Pohjanpelto-09}, the higher order differential invariants are obtained by invariant differentiation of the non-phantom differential invariants of order $\leq 8$.  These are the order $7$ Chern--Moser invariants $\CMJ,\CMoJ,\CMK$, along with the $11$ order $8$ invariants
\begin{equation}\label{V8}
V_{Z^6\oZ{}^2}, \ V_{Z^5\oZ{}^3}, \ V_{Z^4\oZ{}^4}, \ V_{Z^3\oZ{}^5}, \ V_{Z^2\oZ{}^6}, \ V_{Z^5\oZ{}^2U}, \ V_{Z^4\oZ{}^3U}, \ V_{Z^3\oZ{}^4U}, \ V_{Z^2\oZ{}^5U}, \ V_{Z^4\oZ{}^2U^2}, \ V_{Z^2\oZ{}^4U^2}.
\end{equation}
However the syzygies among the differentiated invariants resulting from the recurrence formulae imply that this set of generators is redundant, and we can get by with a much smaller number.

\begin{Remark}
In principle, using the methods of \cite{Olver-Pohjanpelto-09}, one can determine a finite number of syzygies among the differential invariants that generate all the others.  However, we have not attempted to do this since {\sl a}) the required analysis seems exceedingly complicated and {\sl b}) appears to shed little additional light on their structure.
\end{Remark}

We claim that we can, in all cases, generate all the differential invariants by invariant differentiation of the seventh order Chern--Moser invariants ${\rm Re} \>\CMJ, {\rm Im} \>\CMJ,\CMK$, and the eighth order real differential invariant
\begin{equation}\label{CML}
\CML:=V_{Z^4\oZ{}^4}.
\end{equation}
We do not know if this is a minimal generating set\footnote{An important but difficult question is to determine a minimal set of generating differential invariants for a given (pseudo-) group action.  If the set consists of a single differential invariant, it is obviously minimal.  Otherwise, except in the case of curves (one-dimensional submanifolds), there is no known criterion for determining whether or not a given generating set is minimal.}. Moreover, for suitably generic hypersurfaces, all the differential invariants can be generated from the single  order $7$ differential invariant $\CMK$!  The genericity assumption requires that the hypersurface be ``$\CMK$-nondegenerate'' as per the following definition, which is modeled on the definition of a ``mean curvature degenerate'' surface $S \subset \R^3$, \cite{Olver-2009, Olver-2018}.

\begin{Definition}\label{CMKnondegenerate}
A non-umbilic hypersurface is called \emph{$\CMK$-nondegenerate} if there exist $a,b \in \{Z,\oZ,U\}$ such that
\begin{equation}\label{Knondegenerate}
d\CMK \wedge d(\D_a \CMK) \wedge d(\D_b \CMK) = \det \>\begin{pmatrix}\DZ \CMK  & \DoZ \CMK & \DU \CMK \\\DZ \D_a\CMK  & \DoZ \D_a\CMK & \DU \D_a\CMK \\\DZ \D_b\CMK  & \DoZ \D_b\CMK & \DU \D_b\CMK\end{pmatrix} \omega^Z \wedge \omega^{\oZ}\wedge \omega^U \ne 0.
\end{equation}
\end{Definition}

The condition \eqref{Knondegenerate} is equivalent to the condition that, on the hypersurface, the differential invariants $\CMK,\D_a \CMK,\D_b \CMK$ are functionally independent in a neighborhood of $\bp$, which is thus the generic case.  Thus, a hypersurface is $\CMK$-degenerate when $\{\CMK,\DZ \CMK,\DoZ \CMK,\DU \CMK\}$ contains at most $2$ functionally independent invariants. For example, if $\CMK$ is constant, the hypersurface is $\CMK$-degenerate.

\begin{Theorem}\label{J}
Let $M\subset \mathbb{C}^2$ be a nondegenerate non-umbilic hypersurface. Then  every normalized differential invariant $V_J$ can be written as a polynomial function of the seventh order Chern--Moser invariants ${\rm Re} \>\CMJ, {\rm Im} \>\CMJ,\CMK$, the eighth order real differential invariant $\CML$, and their repeated invariant derivatives, which thus generate the entire differential invariant algebra. Furthermore, if $M$ is $\CMK$-nondegenerate then the differential invariants can be expressed as rational functions of the single real-valued Chern--Moser invariant $\CMK$ and its invariant derivatives, and so, for such hypersurfaces, the differential invariant algebra is generated by a single invariant.
\end{Theorem}

\begin{proof}
The first step is to calculate the recurrence relations \eqref{dVJ} for $d\CMJ,d\CMoJ,d\CMK$, which, in view of \eqref{dVJ}, produces the following formulae for the $9$ differentiated Chern--Moser invariants
\begin{equation}\label{D-CMJK}
\aligned
\DZ\, \CMJ&=\frac{1}{240}\,V_{Z^6\oZ{}^2} + \frac{1}{48}\,\CML -\frac{55}8\,\CMJ ^2, \qquad
\DoZ\, \CMJ=\frac{1}{240}\,V_{Z^5\oZ{}^3}+\frac{1}{48}\,V_{Z^3\oZ{}^5}+\frac 58\,\CMJ \,\CMoJ   +8\,\CMK,
\\
\DU\, \CMJ&=\frac{1}{240}\,V_{Z^5\oZ{}^2U} + \frac{1}{48}\,V_{Z^3\oZ{}^4U}-\frac32\i \CMJ\, \CMK,\quad
\DZ\, \CMoJ =\overline{\DoZ\, \CMJ},\quad
\DoZ\, \CMoJ =\overline{\DZ\, \CMJ},\quad
\DU\, \CMoJ =\overline{\DU\, \CMJ},
\\
\DZ\, \CMK&=- \frac{\i}{96} V_{Z^5\oZ{}^2U} + \frac{\i}{96} V_{Z^3\oZ{}^4U}+\frac{5}{192}\,\CMJ \,V_{Z^5\oZ{}^3}+\frac{5}{192}\,\CMoJ \,\CML  + 5\,\CMJ\, \CMK, \qquad
\DoZ\, \CMK=\overline{\DZ\, \CMK},
\\
\DU \CMK&=-\frac{\i}{96} V_{Z^4\oZ{}^2U^2} + \frac{\i}{96} V_{Z^2\oZ{}^4U^2}+\frac{5}{192}\,\CMJ \, V_{Z^4\oZ{}^3U}+\frac{5}{192}\,\CMoJ \,  V_{Z^3\oZ{}^4U}.
\endaligned
\end{equation}
Thus, we can write $9$ of the order $8$ differential invariants \eqref{V8} as polynomial functions involving these first order differentiated invariants and the remaining $2$ differential invariants which we take to be $\CML = V_{Z^4\oZ{}^4}$ and $V_{Z^2\oZ{}^4U^2}$.  In fact, the only differential invariant that explicitly requires $V_{Z^2\oZ{}^4U^2}$ is its complex conjugate $V_{Z^4\oZ{}^2U^2}$; all the others can be written in terms of $\CMJ , \CMoJ   ,\CMK$, their derivatives, and $\CML $.

To write down the remaining invariant $V_{Z^2\oZ{}^4U^2}$, we need to look at the non-phantom recurrence relations \eqref{rec-rel} for the $8$-th order invariants.  Owing to \eqref{varpi}, these will also involve the $11$ non-phantom basic $9$-th order differential invariants $V_J$, $\#J = 9$; however, taking suitable combinations of the recurrence relations will eliminate any $9$-th order terms and produce a syzygy among the $7$-th and $8$-th order differential invariants.  In particular, consider
\begin{equation}\label{Delta}
\Delta = \DoZ V_{Z^3\oZ{}^4U} - \DZ V_{Z^2\oZ{}^5U} .
\end{equation}
By the preceding result, we can write $V_{Z^3\oZ{}^4U}$ and $V_{Z^2\oZ{}^5U}$ as polynomials involving $\CMJ, \CMoJ,\CMK, \CML$, and their derivatives, and hence, by differentiating these polynomials, $\Delta $ is itself such a polynomial combination\footnote{This and subsequent formulae were obtained with the help of Mathematica.  They are quite complicated and we have chosen not to write them down here.  Details are available from the authors upon request.}. On the other hand, we can substitute the formulae for $\DoZ V_{Z^3\oZ{}^4U}$ and $\DZ V_{Z^2\oZ{}^5U}$ coming from their recurrence formulae into $\Delta $; the result is a polynomial in the basic differential invariants of order $7$ and $8$.  Substituting our previous expressions for the $9$ eighth order invariants produces an expression that contains a constant multiple of $V_{Z^4\oZ{}^2U^2}$ plus a polynomial in $\CMJ , \CMoJ   ,\CMK$, their derivatives, and $\CML $.  By equating these two formulae for $\Delta $, we can solve for $V_{Z^4\oZ{}^2U^2}$ as a polynomial function of the four basic differential invariants $\CMK, \CMJ, \CMoJ, \CML$ and their  derivatives. This completes the proof of the first part of the theorem.

To prove the second statement for $\CMK$-nondegenerate hypersurfaces, we recall the ``commutator trick'' introduced in \cite{Olver-2009}.  We begin by substituting the preceding normalization formulae for the Maurer--Cartan forms into the structure equations \eqref{rrihf} for the invariant horizontal forms\footnote{As before, we suppress any contributions from contact forms.}
\begin{equation*}
\begin{aligned}
d\omega^Z&=\frac{5}{8}\,\CMoJ \,\omega^Z\wedge\omega^{\oZ}-\i\left(\frac1{96}V_{Z^5\oZ{}^3}+4\,\CMK \right)\,\omega^Z\wedge\omega^U- \frac\i{96}\,\CML \,\omega^{\oZ}\wedge\omega^U,
\qquad d\omega^{\oZ}&=\overline{d\omega^Z},
\\
d\omega^U&=2\i \omega^Z\wedge\omega^{\oZ}+\frac{5}{4} \,\CMJ \, \omega^Z\wedge\omega^U+\frac{5}{4}\, \CMoJ \, \omega^{\oZ}\wedge\omega^U.
\end{aligned}
\end{equation*}
The coefficients appearing on the right-hand side are known as the \emph{commutator invariants} since they form the coefficients in the formulae for the commutators of the invariant differential operators
\begin{equation}\label{commutators}
\aligned
&[\DZ, \DoZ]=-\frac{5}{8}\,\CMoJ \,\DZ+\frac{5}{8}\,\CMJ \,\DoZ-2\i\DU,
\\
&[\DZ, \DU]=\i\left(\frac1{96}V_{Z^5\oZ{}^3}+4\,\CMK \right)\DZ - \frac\i{96}\,\CML \,\DoZ-\frac{5}{4}\,\CMJ \,\DU,\\
&[\DoZ, \DU]=\frac\i{96}\,\CML \,\DZ-\i\left(\frac1{96}V_{Z^3\oZ{}^5}+4\,\CMK \right)\DoZ-\frac{5}{4}\,\CMoJ \,\DU.
\endaligned
\end{equation}
Now if we apply the second commutator identity in \eqref{commutators} to the differential invariant $\CMK $ and its derivatives, we obtain the following syzygies
\begin{equation}\label{commutator1}
\aligned
&[\DZ, \DU]\,\CMK=\i\left(\frac1{96}V_{Z^5\oZ{}^3}+4\,\CMK \right)\DZ\CMK - \frac\i{96}\,\CML \,\DoZ\CMK-\frac{5}{4}\, \CMJ \,\DU\CMK,\\
&[\DZ, \DU]\,\D_a\CMK=\i\left(\frac1{96}V_{Z^5\oZ{}^3}+4\,\CMK \right)\DZ\D_a\CMK - \frac\i{96}\,\CML \,\DoZ\D_a\CMK-\frac{5}{4}\,\CMJ \,\DU\D_a\CMK,\\
&[\DZ, \DU]\,\D_b\CMK=\i\left(\frac1{96}V_{Z^5\oZ{}^3}+4\,\CMK \right)\DZ\D_b\CMK - \frac\i{96}\,\CML \,\DoZ\D_b\CMK-\frac{5}{4}\,\CMJ \,\DU\D_b\CMK,\\
\endaligned
\end{equation}
where $a,b \in \{Z,\oZ,U\}$.  We treat \eqref{commutator1} as a system of inhomogeneous linear algebraic equations for the three commutator invariants $\frac\i{96}V_{Z^5\oZ{}^3}+4\i\CMK,\> -\frac\i{96}\,\CML ,\>-\frac{5}{4}\,\CMJ$, whose coefficients and inhomogeneous terms involve only $\CMK $ and its iterated invariant derivatives (up to degree $3$).  Thus, if the $\CMK$-nondegeneracy condition \eqref{Knondegenerate} holds, then we can solve this system of linear equations to express the commutator invariants, and hence $\CMJ,V_{Z^5\oZ{}^3},\CML $, as rational combinations of $\CMK $ and its derivatives, the denominator being the nonzero determinant in \eqref{Knondegenerate}.  A similar argument applied to the third operator identity in \eqref{commutators} allows us to similarly express $\CMoJ,V_{Z^3\oZ{}^5},\CML $, the latter having thus two different such formulae\footnote{In general, differentiated invariants can have many such formulae owing to the variety of syzygies among the differential invariants.}.
\end{proof}

An alternative approach to the second result that avoids the commutator trick, and thereby reveals some additional interesting structure arising from the recurrence formulae, proceeds as follows.  Let us compute all $27$ second derivatives $\D_a \D_b \CMJ, \D_a \D_b \CMoJ,\D_a \D_b  \CMK$ by differentiating \eqref{D-CMJK}.  Replacing the 10 eighth order differential invariants by the preceding expressions and then eliminating the $11$ ninth order basic differential invariants (which occur linearly) from the resulting identities, leads to $11$ independent syzygies.  The eighth order invariant $\CML $ occurs linearly in these identities, but its coefficients depend on $\CMJ,\CMoJ,\CMK$ and their first order derivatives.  Thus, one can only solve for $\CML $ in terms of $\CMJ,\CMoJ,\CMK$ and their derivatives if at least one of these coefficients is nonzero.  Closer inspection reveals the following

\begin{Theorem}
For a general nondegenerate hypersurface in $\mathbb C^2$, if any one of the following invariants does not vanish
\begin{equation}\label{DZJK24}
\DZ \,\CMJ,  \quad \DoZ\, \CMJ, \quad \DZ\, \CMK,  \quad \DoZ\, \CMK, \quad \ \DoZ\, \CMJ + \frac{24}{5} \left(\CMK - \frac{25}{24} \,|\CMJ|^2\right),
\end{equation}
then one can write $\CML $ rationally in terms of $\CMJ,\CMoJ,\CMK$ and their first and second order invariant derivatives with one of these quantities appearing in the denominator\footnote{If more than one is nonzero, then one can construct several such expressions; their equivalence is a consequence of the various syzygies among $\CMJ,\CMoJ,\CMK$ and their invariant derivatives.}.
\end{Theorem}

The most interesting is the last quantity in \eqref{DZJK24}.  If $\CMJ,\CMoJ,\CMK$ are all constant, then the  syzygy that it appears in takes the reduced form
\begin{equation}\label{2524}
\left(\CMK - \frac{25}{24} \,|\CMJ|^2\right)\left(\CML  - \frac{75}2 \,{\rm Re}\> ( J^2)\right) = 192.
\end{equation}
This implies that the initial factor cannot vanish: $\CMK \ne \frac{24}{25} \,|\CMJ|^2$ when both are constant. Under this assumption, one can then express $\CML $ as an explicit rational function of $\CMJ,\CMoJ,\CMK$, which in particular implies that it is also constant.  We conclude that if $\CMJ,\CMoJ,\CMK$ are constant, then so are all the higher order differential invariants, which implies that the hypersurface possesses a three-dimensional symmetry group and hence is a subset of one of its orbits in $\C^2$. See Section \ref{sec-iso} for a complete classification of such ``maximally symmetric'' hypersurfaces. We find the required non-vanishing constraint among the constant differential invariants quite surprising, and worthy of further investigation as to its significance.  On the other hand, if $\CMJ$ and $\CMK$ are not constant, the corresponding syzygy just imposes one more equation relating their derivatives, and does not produce any non-vanishing constraint.

If all the quantities in \eqref{DZJK24} vanish and the hypersurface is $\CMK$-degenerate, then it appears that all $4$ differential invariants  $\CMJ,\CMoJ,\CMK,\CML $ are required to generate the differential invariant algebra.  On the other hand, we cannot completely rule out the existence of an even higher order syzygy that can be solved for $\CML $ without any restrictions on $\CMJ,\CMoJ,\CMK$, although, based on what we have been able to compute, this seems highly unlikely.  Given the difficulty of the preceding computations, rigorously establishing this assertion would be quite challenging.

\subsection{Equivalence and rigidity of non-umbilic hypersurfaces}
\label{sec-equi-rigid}

According to Cartan --- see \cite{Olver-Fels-99,Olver-1995} --- under suitable regularity conditions, two analytic submanifolds are locally congruent (equivalent) under the action of a pseudo-group if and only if all differential invariants, when evaluated thereon, have identical functional interrelationships (syzygies). Since, by the Lie--Tresse Theorem, \cite{Kruglikov-Lychagin-16}, all higher order differential invariants can be obtained by invariantly differentiating a finite collection of generating invariants, it suffices to check the syzygies among a suitable finite collection, which will parametrize the differential invariant signature (or classifying manifold) that characterizes submanifolds up to local congruence; see \cite{Olver-1995} for details.   The specification of which differential invariants are required to construct the signature depends on the underlying structure of the differential invariant algebra and the order at which the submanifold jet becomes nonsingular. We shall, in addition, assume that the submanifolds are \emph{signature regular}, meaning that their differential invariant signature forms a manifold of a fixed dimension, or, equivalently, the number of functionally independent differential invariants does not vary. Two signature regular submanifolds are locally congruent if and only if their signatures are the same when evaluated on the corresponding open subsets.

In our case, since the hypersurface $M$ has dimension $3$, under the Cartan regularity assumption, the number of functionally independent differential invariants on it, or, equivalently, the dimension of the signature, is (locally) either $0,1,2$ or $3$. The first case occurs when all the differential invariants are constant.  Now since for any non-umbilic hypersurface, the differential invariant algebra is generated by the Chern--Moser invariants\footnote{One should keep in mind that the signature so constructed retains the sign ambiguities \eqref{jksign}.  These can be eliminated by replacing $\CMJ$ and $ V_{Z^j\oZ{}^kU^\ell}$ for $j+k$ odd by $\CMJ^2$ and $\CMJ \, V_{Z^j\oZ{}^kU^\ell}$, respectively, when parametrizing the signature.} ${\rm Re} \ \CMJ, {\rm Im} \ \CMJ,\CMK, \CML$, the signature has differential invariant order $7 \leq k \leq 10$.  The lowest order $k=7$ is when $\CMJ, \CMK, \CML$ are constant and the hypersurface has a three-dimensional local symmetry group.
The maximal order $k=10$ occurs when there are $3$  functionally independent differential invariants, only one of which can be found among ${\rm Re}\>\CMJ, {\rm Im}\>\CMJ,\CMK$; a second among their invariant derivatives of degree $1$ and $\CML$, and a third among the latters' invariant derivatives (which are differential invariants of order $9$), in which case the syzygies of order $\leq 10$ are required to completely determine all those of higher order.

These considerations immediately produce a characterization of the rigidity properties of hypersurfaces under holomorphic maps.  First, recall that two submanifolds $M, \widetilde M$ are said to have \emph{order $k$ contact} at a common point $\bp \in M \cap \widetilde M$ if they have the same $k$-th order jet there.   In our situation, the contact and rigidity properties rely on the jet coordinates provided by the partial derivatives of the defining function $v = f(z,\oz,u)$ with respect to its three arguments. The following definition is based on \cite{Olver-Fels-99}; see also \cite{Jensen}.

\begin{Definition}
A nonsingular hypersurface $M$ is called \emph{order $k$ rigid} if, whenever $\widetilde M$ and $M$ have order $k$ contact at a common point $\bp \in M \cap \widetilde M$ and $\widetilde M = g \cdot M$ for some holomorphic transformation $g \in \G$, then necessarily $\widetilde M = M$.
\end{Definition}

\begin{Theorem}
\label{rigid}
A non-umbilic hypersurface $M \subset \mathbb C^2$ that is both signature regular and fully regular at order $\leq 10$ has rigidity order at most $10$.
\end{Theorem}

In other words, if the normal forms for both $M$ and $g \cdot M$ at a common point $\bp \in M \cap (g \cdot M)$ are the same up to and including the order $10$ terms, then $M = g \cdot M$.  This follows immediately from \cite[Theorem 14.13]{Olver-Fels-99} using the fact that the pseudo-group acts freely on the fully regular subset \eqref{fully regular}.

We remark that the assertion of Theorem \ref{rigid} is quite different from the known  ``finite jet determination'' result of Chern--Moser which states that every CR diffeomorphism between two nondegenerate hypersurfaces is determined uniquely by its jets of order $\leq 2$ at a certain point. See \cite[Theorem 1]{BER-98} for a generalization of this result.

\subsection{Umbilic hypersurfaces}\label{umb}

As noted at the end of Section \ref{sec-NF-general-comp}, a surface is umbilic in a neighborhood of $\bp$ if $V_{Z^4\oZ{}^2}\equiv 0$ vanishes there.  In this case, the second recurrence relation in \eqref{eq: rec-order-6} then reduces to
\[
0 \equiv \varpi_{Z^4\oZ{}^2} = V_{Z^5\oZ{}^2}\,\omega^Z + V_{Z^4\oZ{}^3}\,\omega^{\oZ} + V_{Z^4\oZ{}^2U}\,\omega^U,
\]
which implies that
\begin{equation}\label{eq: order 7 vanishing invariants}
V_{Z^5\oZ{}^2}\equiv V_{Z^4\oZ{}^3} \equiv V_{Z^4\oZ{}^2U}\equiv 0.
\end{equation}
The recurrence relations for the vanishing invariants \eqref{eq: order 7 vanishing invariants} imply that $V_{Z^{\alpha} \oZ{}^{\beta} U^\gamma} \equiv 0$ when $\alpha+\beta+\gamma = 8$.  An inductive argument, combined with our previous normalizations, shows that $V_J \equiv 0$ for all $J$ except, of course, $V_{Z\oZ} = 1$, and hence the normal form at an umbilic point collapses to a single term.  Thus, in the locally umbilic case, none of the five Maurer--Cartan forms \eqref{MC-5} can be fully normalized, which reflects the fact that the isotropy group at the umbilic point is five-dimensional.  We have thus re-established the well-known result that an umbilic hypersurface is locally biholomorphically equivalent to the Heisenberg sphere, \cite[Theorem 7.1]{Jacobowitz}.

\begin{Theorem}
Let $M\subset\mathbb C^2$ be umbilic in a neighborhood of $\bp$. Then, there exists a holomorphic map of $\mathbb C^2$ transforming $M$ locally to the Heisenberg sphere
\begin{equation}\label{Hsphere}
\mathbb H: \qquad v=z\oz.
\end{equation}
This transformation is unique up to the action of the five-dimensional isotropy group whose infinitesimal generators are the holomorphic vector fields $\vv_1, \ldots, \vv_5$ introduced in \eqref{G8}.
\end{Theorem}

\begin{Remark}
As shown by Chern and Moser, \cite{Chern-Moser}, the ordinary sphere $\mathbb S^3 = \{ |z|^2 + |w|^2 = 1\}$ is locally congruent to the Heisenberg sphere, and globally congruent if one removes a single point.  More generally, using the results in \cite[Theorem 14.30]{Olver-1995}, any connected hypersurface that is everywhere locally equivalent to the Heisenberg sphere must be globally equivalent modulo generalized covering maps, meaning surjective maps that are everywhere local diffeomorphisms.  This raises the interesting question of geometrically characterizing when such covering hypersurfaces can be embedded in $\C^2$.  A similar observation holds when a non-umbilic hypersurface has all constant differential invariants and so is locally equivalent to an orbit $\mathcal O$ of a three-dimensional symmetry group, and globally equivalent up to covering.
\end{Remark}

\subsection{Singularly umbilic points --- preliminaries}
\label{sec-semi-umb}

Let us now turn our attention to the previously unstudied case of a singularly umbilic point $\bp \in M$, meaning a point at which the Cartan curvature vanishes, but is not identically zero in a neighborhood thereof.  In this case, the partial normal form \eqref{NF-pre} contains no sixth order terms, but, unlike the locally umbilic case, it contains nonzero terms of higher order.  We use the non-vanishing term(s) of lowest order to specify the type of the singularly umbilic point.

\begin{Definition}\label{singularly-umbilic-def}
The singularly umbilic point $\bp\in M \subset \mathbb{C}^2$ is said to be of \emph{umbilic order} $n_0>6$ if there exists $\alpha,\beta,\gamma\in \N_0$ such that
\begin{itemize}
\item $n_0=\alpha+\beta+\gamma > 6$, with $ \alpha \geq 4$, $\beta \geq 2$, $\gamma \geq 0$, and $\alpha  \geq \beta $;
\item $V_{Z^\alpha  \oZ{}^\beta  U^\gamma }|_{\bp}\neq 0$;
\item $n_0$ is minimal in the sense that $V_{Z^j \oZ{}^k U^\ell}|_{\bp}=0$ at $\bp$ for all $j, k, \ell\in \N_0$ with $6 \leq j+k+\ell<n_0$.
\end{itemize}
The triple $(\alpha,\beta,\gamma)$ is called the \emph{umbilic type} of the point $\bp\in M$.
\end{Definition}

\begin{Remark}
We note that the notion of umbilic order is uniquely defined, while the umbilic type $(\alpha,\beta,\gamma)$ is not necessarily unique. Indeed, it is conceivable that there exist two triples $(\alpha,\beta,\gamma)$ and $(\alpha^\prime,\beta^\prime,\gamma^\prime)$ of the same order $n_0 = \alpha+\beta+\gamma = \alpha^\prime+\beta^\prime+\gamma^\prime$ such that $V_{Z^{\alpha} \oZ{}^{\beta} U^\gamma}|_{\bp}\neq 0$ and $V_{Z^{\alpha^\prime} \oZ{}^{\beta^\prime} U^{\gamma^\prime}}|_{\bp}\neq 0$.  In particular, If $V_{Z^{\alpha} \oZ{}^{\beta} U^\gamma}|_{\bp}\neq 0$, then its complex conjugate $V_{Z^{\beta} \oZ{}^{\alpha} U^\gamma}|_{\bp}\neq 0$ also, which explains our extra condition $\alpha \geq \beta $ when specifying the type.  If $\alpha=\beta $ then these two lifted invariants are equal and real.
\end{Remark}

We use these observations to impose a refinement of the notion of a singularly umbilic point.

\begin{Definition}
Suppose that $\bp \in M$ is a singularly umbilic point of order $n_0$.  Then we call it
\begin{itemize}
  \item[$\bullet$] \emph{generic}\footnote{One should not confuse this definition of generic hypersurface $M$ with the standard definition of generic manifold in CR geometry.} if it has an umbilic type $(\alpha,\beta,\gamma)$ with $\alpha > \beta $.
  \item[$\bullet$] \emph{semi-circular} if a) all umbilic types $(\alpha,\beta ,\gamma)$ of $\bp$ satisfy $\alpha=\beta$, and b) there exists a lifted invariant of order $\sctypez + \sctypeoz + \sctypeu > n_0$ with $\sctypez \ne \sctypeoz $ such that $V_{Z^\sctypez \oZ{}^\sctypeoz U^\sctypeu}|_{\bp}\neq 0$. We call $(\sctypez ,\sctypeoz ,\sctypeu )$ the \emph{semi-circular type} of $\bp$.
  \item[$\bullet$] {\it circular} otherwise. In other words, if $\bp$ is a circular singularly umbilic point, the only potentially nonzero monomials $z^j\oz^k u^\ell$ in the corresponding partial normal form \eqref{NF-pre} are those for which $j = k$, and hence the Taylor expansion has the form
\begin{equation}
\label{circularnf}
v=z\oz+\sum_{2k+\ell \geq n_0 \geq 8 \atop k \geq 4}\frac{V_{Z^k\oZ{}^kU^\ell} }{k!^2\,\ell!} (z\oz)^ku^\ell.
\end{equation}
\end{itemize}
\end{Definition}

We observe that the circular normal form \eqref{circularnf} admits the rotational automorphism $z \longmapsto e^{\i t} z$ generated by the ``circular'' vector field $\i z \, \partial _z$.

Thus, to deal with the singularity at order 6 caused by the vanishing of $V_{Z^4\oZ{}^2}$ at $\bp$, we pass to the order $n_0$ prolonged action, which is where the next non-vanishing partially normalized differential invariant appears.

\subsection{Generic singularly umbilic points}
\label{sec-generic-semi-umb}

Let us first treat the case of a generic singularly umbilic point, where
 $\bp$ has umbilic type $(\alpha,\beta,\gamma)$ and order $n_0=\alpha+\beta+\gamma > 6$ with $\alpha > \beta $.  Then we can impose the two independent normalizations
\begin{equation}
\label{V4a2bg}
V_{Z^{\alpha} \oZ{}^{\beta} U^\gamma}=V_{Z^{\beta} \oZ{}^{\alpha} U^\gamma}=\sigma ,
\end{equation}
where $\sigma$ is a {\it nonzero} real constant whose value will be specified later.  As we will see, one can then use the two resulting phantom recurrence formulae
\begin{equation}
\label{V4a2bg0}
dV_{Z^{\alpha} \oZ{}^{\beta} U^\gamma}=dV_{Z^{\beta} \oZ{}^{\alpha} U^\gamma}=0
\end{equation}
to normalize the Maurer--Cartan forms $\mu_Z$ and $\overline\mu_{\oZ}$.

Before continuing, we need to analyze in more details the recurrence formulae obtained by invariantization of the prolonged infinitesimal generator coefficients.  Suppose $n = j+k+\ell \geq 2$. The invariantization of formula \eqref{phijkl} for the infinitesimal generator coefficients takes the form
\begin{equation}
\label{iphijkl}
\aligned
&\phi^{z^j\oz^k u^{\ell}}(\bZ\n,\bmu\n)=-V_{Z^j\oZ{}^k U^\ell}\,\bigl[j\,\mu_{Z}+k\,\overline\mu_{\oZ}+(\ell-1)\,\mcu_U\bigr]
-\ell\,V_{Z^{j+1}\oZ{}^k U^{\ell-1}} \,\mu_{U}\\
& \hskip1.6in {}-\ell\,V_{Z^{j}\oZ{}^{k+1} U^{\ell-1}}\,\overline\mu_{U} - \ell\,V_{Z^j\oZ{}^k U^{\ell-1}}\left[j\,\mu _{ZU}+k\,\overline\mu_{{\oZ}U}+\frac{\ell-3}2\,\mcu_{UU}\right] 
\\& \hskip1.6in {} + \i{\tt N}_{jk\ell}\,\mu_{U} - \i\overline{{\tt N}_{kj\ell}}\,\overline\mu_{U} - \i{\tt P}_{jk\ell}\,\mu _{ZU}+  \i\overline{{\tt P}_{kj\ell}}\,\overline\mu_{{\oZ}U} + {\tt Q}_{jk\ell}\,\mcu_{UU}+\>\cdots,
\endaligned
\end{equation}
where ${\tt N}_{jk\ell},{\tt P}_{jk\ell},{\tt Q}_{jk\ell}$ are the invariantizations of the differential polynomials $N_{jk\ell},P_{jk\ell}, Q_{jk\ell}$, and hence depend superquadratically on the lifted invariants $V_J$ for $1 \leq \# J \leq j+k+\ell$.  The omitted terms in \eqref{iphijkl} are linear combinations of the Maurer--Cartan forms that do not appear among the displayed terms. We substitute this formula into the recurrence relation for the lifted differential invariant  $V_{Z^j\oZ{}^k U^\ell}$ and simplify using the normalization formulae for the Maurer--Cartan forms in Proposition \ref{prop-gen-patt}.  In particular, Remark \ref{rem} implies that, as a result, the omitted terms only produce multiples of $\mu _U$ and $\overline\mu_U$, modulo horizontal forms.  The resulting recurrence relation thus has the form
\begin{equation}
\label{V-ord-k}
\aligned
dV_{Z^j\oZ{}^k U^\ell}&\equiv (1-j-\ell) \,V_{Z^j\oZ{}^k U^\ell}\,\mu_{Z}+(1-k-\ell)\,V_{Z^j\oZ{}^k U^\ell}\,\overline\mu_{\oZ} + \bigl[-\ell\,V_{Z^{j+1}\oZ{}^k U^{\ell-1}} + {\tt A}_{jk\ell}\bigr]\,\mu_{U}
\\& \qquad \quad {}+ \bigl[-\ell\,V_{Z^{j}\oZ{}^{k+1} U^{\ell-1}}+ {\tt B}_{jk\ell}\bigr]\,\overline\mu_{U}+ \bigl[\ell\,(3-j-k-\ell)\,V_{Z^j\oZ{}^k U^{\ell-1}}+ {\tt C}_{jk\ell}\bigr]{\rm Re} \>\mu_{ZU},
\endaligned
\end{equation}
where ${\tt A}_{jk\ell}, {\tt B}_{jk\ell},{\tt C}_{jk\ell}$ are polynomials depending on the lifted invariants $V_J$.
Keep in mind that all the lifted invariants $V_J$ for $\#J < n_0$ vanish at $\bp$ except for $V_{Z\oZ}=1$, and that we can only make use of the  recurrence formulae for $V_K$ when either $j \geq 2$ and $k \geq 4$, or $j \geq 4$ and $k \geq 2$, since the others have already been used in our earlier normalizations.  Now for $K=(j, k, \ell)$ with $\#K = n_0$, all the terms in the invariantization of the infinitesimal generator coefficient $\phi ^K$  vanish at $\bp$ with the possible exceptions of a) the terms involving $V_J$ with $\#J = n_0$; these terms are necessarily linear in $V_J$ and appear explicitly in \eqref{iphijkl}, and b) those involving a power of $V_{Z\oZ}$ and no other $V_J$'s.
 Using \eqref{monomials}, let us analyze the terms involving pure powers $V_{Z\oZ}^m$ for $m\in \N$.   The coefficient of the first power is
$$j\, k\, \mcv_{Z^{j-1}\oZ{}^{k-1}U^\ell V} - k\, \mu_{Z^{j}\oZ{}^{k-1}U^\ell}- j\, \omu_{Z^{j-1}\oZ{}^{k}U^\ell},$$
where $\mcv$ is the Maurer--Cartan form introduced in \eqref{eq: substitutions}.
In view of the linear relations \eqref{eq: MC relations} and the formulae in Proposition \ref{prop-gen-patt}, all three terms are horizontal at $\bp$ unless $j=1$ or $k=1$, which is not permitted by the preceding remark.
Next the coefficient of $V_{Z\oZ}^2$ is
$$m_K\,\mcv_{Z^{j-2}\oZ{}^{k-2}U^\ell V^2} - n_K\, \mu_{Z^{j-1}\oZ{}^{k-2}U^\ell V}- n_K\, \omu_{Z^{j-2}\oZ{}^{k-1}U^\ell V},$$
where $m_K,n_K \in \N$. Again applying \eqref{eq: MC relations} and Proposition \ref{prop-gen-patt}, all three terms are horizontal at $\bp$ unless $j=k=2$, or $j=2,\ k=3$, or $j=3,\ k=2$, but these are again not permitted.
A similar analysis shows that the coefficients of any higher power $V_{Z\oZ}^m$ for $m \geq 3$ are also horizontal at $\bp$.  We conclude that, when $\#K = n_0$, only the terms explicitly displayed in \eqref{iphijkl} produce nonzero multiples of the as yet unnormalized Maurer--Cartan forms $\mu_Z, \omu_{\oZ}, {\rm Re}\> \mu_{ZU},\mu_{U},\overline\mu_U$.
A similar  argument also applies when $\#K = n_0+1$, which will be needed when we normalize ${\rm Re} \>\mu_{ZU}$.

We conclude that, when $j+k+\ell = n_0$ or $n_0+1$, the polynomials ${\tt A}_{jk\ell}, {\tt B}_{jk\ell},{\tt C}_{jk\ell}$ vanish at the point $\bp$, and formula \eqref{V-ord-k} reduces to
 \begin{equation}
\label{V-ord-k0}
\aligned
dV_{Z^j\oZ{}^k U^\ell}|_{\bp}&\equiv (1-j-\ell) \,(V_{Z^j\oZ{}^k U^\ell}\,\mu_{Z})|_{\bp}+(1-k-\ell)\,(V_{Z^j\oZ{}^k U^\ell}\,\overline\mu_{\oZ})|_{\bp} -\ell\,(V_{Z^{j+1}\oZ{}^k U^{\ell-1}}\,\mu_{U})|_{\bp}\\& \qquad \qquad {}-\ell\,(V_{Z^{j}\oZ{}^{k+1} U^{\ell-1}}\,\overline\mu_{U})|_{\bp}+ \ell\,(3-j-k-\ell)\,(V_{Z^j\oZ{}^k U^{\ell-1}}\,{\rm Re} \>\mu_{ZU})|_{\bp}.
\endaligned
\end{equation}
If $j+k+\ell = n_0$, the coefficient $V_{Z^j\oZ{}^k U^{\ell-1}}|_{\bp}$ of ${\rm Re} \>\mu_{ZU}|_{\bp}$ vanishes since it has order $n_0-1$.

Now, normalizing as in \eqref{V4a2bg}, we solve the associated phantom recurrence formulae \eqref{V4a2bg0} for
 \begin{equation}
\label{muz}
\mu_Z \equiv  {\tt A}\, \mu _U+ {\tt B}\,\overline\mu_U  + {\tt C}\,{\rm Re} \>\mu_{ZU},\qquad
\overline\mu_{\oZ} \equiv \overline {\tt B}\, \mu _U+ \overline {\tt A}\,\overline\mu_U + \overline {\tt C}\, {\rm Re} \>\mu_{ZU},
\end{equation}
where, in view of  \eqref{V-ord-k}, ${\tt A},{\tt B},{\tt C}$ are polynomials in $V_J$.  Moreover, according to \eqref{V-ord-k0} and the remark immediately following it, they have the following values at $\bp$
\begin{equation}
\label{ABC0}
\aligned
{\tt A}|_{\bp}&= \frac{\gamma\,\bigl[(1-\alpha-\gamma) \,V_{Z^{\alpha+1} \oZ{}^{\beta} U^{\gamma-1}}|_{\bp} - (1-\beta-\gamma) \,V_{Z^{\beta +1} \oZ{}^{\alpha} U^{\gamma-1}}|_{\bp}\bigr]}{\sigma \,(\alpha-\beta)\,(\alpha + \beta + 2\,\gamma -2)}, \\
{\tt B}|_{\bp}&=\frac{\gamma\,\bigl[(1-\alpha-\gamma) \,V_{Z^{\alpha} \oZ{}^{\beta+1} U^{\gamma-1}}|_{\bp} - (1-\beta-\gamma) \,V_{Z^{\beta} \oZ{}^{\alpha+1} U^{\gamma-1}}|_{\bp}\bigr]}{\sigma \,(\alpha-\beta )\,(\alpha + \beta + 2\, \gamma -2)}, \qquad {\tt C}|_{\bp}= 0.
\endaligned
\end{equation}
In particular, if $\gamma = 0$, then ${\tt A}|_{\bp}={\tt B}|_{\bp}=0$, and hence $\mu_Z|_{\bp}, \overline\mu_{\oZ}|_{\bp}$ are horizontal in this case.

The next step is to normalize ${\rm Re} \>\mu_{ZU}$, and to do this we set
\[
{\rm Re} \>V_{Z^{\alpha}\oZ{}^{\beta} U^{\gamma+1}}=0.
\]
Substituting this and our earlier normalizations into \eqref{V-ord-k} produces an equation of the form
\begin{equation}
\label{ReVabg}
\aligned
0 = d\,{\rm Re} \>V_{Z^{\alpha}\oZ{}^{\beta} U^{\gamma+1}}\equiv {\tt D}\,\mu_Z + \overline{\tt D}\,\overline\mu_{\oZ} + {\tt E}\, \mu _U + \overline{\tt E}\,\overline\mu_U  + {\tt F}\,{\rm Re} \>\mu_{ZU} \equiv  {\tt G}\,\mu _U + \overline{\tt G}\>\overline\mu_U + {\tt H}\,  {\rm Re} \>\mu_{ZU},
\endaligned
\end{equation}
where the second version is obtained by substituting the formulae \eqref{muz} for $\mu_Z ,\overline\mu_{\oZ}$ into the first, so
$${\tt G} = {\tt E} + {\tt A}\,{\tt D}+ \overline {\tt B}\,\overline {\tt D}, \qquad {\tt H} = {\tt F} + 2\,{\rm Re} \>({\tt C}\,{\tt D}).$$
In particular, according to \eqref{V-ord-k0}, \eqref{ABC0} and the assumption that $\sigma \ne 0$,
\begin{equation*}
\aligned
{\tt H}|_{\bp} = {\tt F}|_{\bp}= - \,(\gamma + 1)(\alpha + \beta + \gamma -2)\, \sigma \ne 0,
\endaligned
\end{equation*}
and hence the lifted invariant polynomial ${\tt H}$ is nonzero in some neighborhood of $\bp$.  Thus, we can use \eqref{ReVabg} to normalize the Maurer--Cartan form
\begin{equation*}
\aligned
{\rm Re} \>\mu_{ZU} \equiv  {\tt R}\, \mu _U + \overline{\tt R}\,\overline\mu_U,
\endaligned
\end{equation*}
where ${\tt R}$ is a rational function of the lifted invariants $V_J$ whose denominator does not vanish at  $\bp$.

At this stage, the only Maurer--Cartan forms that remain to be normalized are $\mu_{U}$ and $\overline\mu_U$. These can be fixed by normalizing
\[
V_{Z^{\alpha-1} \oZ{}^{\beta} U^{\gamma+1}} = V_{Z^{\beta} \oZ{}^{\alpha-1} U^{\gamma+1}}=0.
\]
Using  \eqref{V4a2bg}, \eqref{muz}, \eqref{ABC0}, the corresponding phantom recurrence formulae are reduced to ones of the form
\begin{equation}
\label{muu0}
0 = dV_{Z^{\alpha-1} \oZ{}^{\beta} U^{\gamma+1}} \equiv {\tt S}\, \mu_U + {\tt T}\,\overline\mu_U, \qquad 0 = dV_{Z^{\beta} \oZ{}^{\alpha-1} U^{\gamma+1}} \equiv \overline{\tt T}\, \mu_U + \overline{\tt S}\,\overline\mu_U,
\end{equation}
where
\begin{equation*}
{\tt S}|_{\bp}= -\,(\gamma + 1) \,\sigma,\qquad
{\tt T}|_{\bp}=- \,(\gamma + 1)\,V_{Z^{\alpha-1} \oZ{}^{\beta+1} U^{\gamma}}|_{\bp}.
\end{equation*}
Thus, provided we choose the real constant $\sigma$ in \eqref{V4a2bg} so that
\begin{equation}
\label{l2}
\sigma  \ne \pm\,\bigl|\,V_{Z^{\alpha-1} \oZ{}^{\beta+1} U^{\gamma}}|_{\bp}\,\bigr|, \qquad \sigma \ne 0,
\end{equation}
we can use \eqref{muu0} to normalize both $\mu _U,\overline\mu_U$ in a suitable neighborhood of  $\bp$.

Now that all the Maurer--Cartan forms have been normalized, the holomorphic isotropy subalgebra at the singularly umbilic point is zero-dimensional, i.e., it is discrete.  We have thus produced a completely normalized normal form there.  As in Theorem \ref{nonumbilic}, convergence of the normal form expansion follows from the convergence of the partial normal form \eqref{NF-pre}.

\begin{Theorem}
 Let $M \subset \mathbb C^2$ be a nondegenerate hypersurface that is generic singularly umbilic at $\bp$ of umbilic type $(\alpha,\beta,\gamma )$ and order $n_0=\alpha+\beta+\gamma >6$. Then, there exists a holomorphic transformation mapping $M$ to the convergent normal form
\begin{equation}
\label{semi-umb-nf}
v=z\oz+
\sum_{\scriptstyle j+k+\ell\geq n_0
\atop{\scriptstyle j\geq 2, \ k \geq 4, \ \ell \geq 0,\ostrut84\atop \scriptstyle {\rm or } \ j\geq 4, \ k \geq 2, \ \ell \geq 0}} \frac{1}{j!\, k! \,\ell!} V_{Z^j \oZ{}^k U^\ell} z^j \oz^k u^\ell
\end{equation}
 with
\[
\aligned
V_{Z^{\alpha} \oZ{}^{\beta} U^\gamma}=V_{Z^{\beta} \oZ{}^{\alpha} U^\gamma}=\sigma , \qquad
V_{Z^{\alpha-1} \oZ{}^{\beta} U^{\gamma+1}} = V_{Z^{\beta} \oZ{}^{\alpha-1} U^{\gamma+1}} = {\rm Re} \>V_{Z^{\alpha}\oZ{}^{\beta}U^{\gamma+1}} = 0,
\endaligned
\]
where $\sigma $ is a nonzero real constant satisfying \eqref{l2}. Moreover, the holomorphic isotropy group of $M$ at $\bp$ is at most discrete.
\end{Theorem}

\begin{Remark}
Although all terms of order $6 \leq n < n_0$ in the singularly umbilic normal form vanish, this only implies that the corresponding differential invariants, including the Cartan curvature, vanish at the point $\bp$, and does not imply that they vanish in a neighborhood thereof.
\end{Remark}

\subsection{Semi-circular points}\label{sec-semi-circular}

Suppose $\bp$ is a singularly umbilic point of umbilic type $(\alpha,\alpha ,\gamma)$ and semi-circular type $(\sctypez ,\sctypeoz ,\sctypeu )$.  In this case, only the real form ${\rm Re} \>\mu_Z$ is normalizable through the recurrence relations of order $n_0$. However, setting ${\rm Im} \>V_{Z^\sctypez \oZ{}^\sctypeoz U^\sctypeu}=\tau $  for suitable $0\neq \tau \in \R$, the resulting phantom recurrence formula allows us to normalize ${\rm Im} \>\mu_Z$.  The remainder of the analysis proceeds in the same fashion as at a generic singularly umbilic point, and the details are left to the reader.  Here we just state the final result.

\begin{Theorem}
 Let $M \subset \mathbb C^2$ be a nondegenerate hypersurface that is singularly umbilic at $\bp$ of umbilic type $(\alpha,\alpha ,\gamma )$, semi-circular type $(\sctypez ,\sctypeoz ,\sctypeu )$ and order $n_0=2\alpha+\gamma \geq 8$. Then, there exists a holomorphic transformation mapping $M$ to the convergent normal form  \eqref{semi-umb-nf} with
 \begin{equation}
\label{semi-circular-n}
\aligned
V_{Z^{\alpha} \oZ{}^{\alpha } U^\gamma}=\sigma , \quad
V_{Z^{\alpha-1} \oZ{}^{\beta} U^{\gamma+1}} = V_{Z^{\beta} \oZ{}^{\alpha-1} U^{\gamma+1}} = {\rm Re} \>V_{Z^{\alpha}\oZ{}^{\beta}U^{\gamma+1}} = 0, \quad {\rm Im} \>V_{Z^\lambda\oZ{}^\mu U^\nu}=\tau ,
\endaligned
\end{equation}
where $\sigma ,\tau $ are nonzero real constants. Moreover, the holomorphic isotropy group of $M$ at $\bp$ is at most discrete.
\end{Theorem}

\subsection{Circular points}
\label{sec-circular}

Finally, let us consider the normal form of a circular singularly umbilic hypersurface, of umbilic type $(\alpha, \alpha, \gamma)$ with $\alpha \geq 4$, and order $n_0=2\alpha+\gamma \geq 8$.

Looking {\it ab initio} at the normalizations made in the generic case, one finds that it is possible to normalize the four Maurer--Cartan forms
${\rm Re} \>\mu_Z, \ {\rm Re} \>\mu_{ZU}, \ \mu_U, \ \overline\mu_U$, by setting
\begin{equation*}
V_{Z^{\alpha}\oZ{}^{\alpha}U^{\gamma}}=1 \qquad {\rm and} \qquad V_{Z^{\alpha}\oZ{}^{\alpha}U^{\gamma+1}}=V_{Z^{\alpha-1}\oZ{}^{\alpha}U^{\gamma+1}}=V_{Z^{\alpha}\oZ{}^{\alpha-1}U^{\gamma+1}}=0.
\end{equation*}
However, the residual rotational isotropy of the circular normal form means that the remaining  Maurer--Cartan form ${\rm Im} \>\mu_{Z}$ cannot be normalized, \cite{Valiquette-SIGMA}.

 \begin{Theorem}
 Let $M$ be a circular hypersurface of $\mathbb C^2$ singularly umbilic at $\bp$ of umbilic type $(\alpha, \alpha, \gamma)$ and order $n_0=2\alpha+\gamma\geq 8$. Then there exists a holomorphic transformation mapping $M$ to the convergent normal form
\begin{equation}
\label{circnf}
 v=z\oz+\frac{1}{(\alpha!)^2\,\gamma !}\,z^{\alpha}\oz^{\alpha}u^\gamma+\sum_{2k+\ell\geq n_0 \atop{k\geq \alpha+1}}\,\frac{1}{(k!)^2 \ell!}\, V_{Z^k \oZ{}^k U^\ell}\, (z\oz)^k u^\ell
\end{equation}
with $V_{Z^{\alpha}\oZ{}^{\alpha}U^{\gamma+1}}=0$. The holomorphic isotropy group of $M$ at  $\bp$ is one-dimensional and, when mapped to the origin, is generated by the rotational vector field $\i z\, \partial _z$.  Moreover, the circular normal form \eqref{circnf} is unique.
 \end{Theorem}

\subsection{Symmetry and Isotropy}
\label{sec-iso}

Before concluding this paper, let us present a classification of nondegenerate hypersurfaces of $\mathbb C^2$ according to their (isotropy) group of holomorphic automorphisms. This classification follows immediately from the above considerations and can be deduced, in part, from  \cite{Beloshapka-80}.

\begin{Theorem}
Let $M$ be a hypersurface of $\mathbb C^2$ nondegenerate at a point $\bp$. Let $\frak G_\bp$ be the  holomorphic isotropy group at $\bp$. Then one of the following possibilities will occur
\begin{itemize}
  \item[$\bullet$] ${\rm dim}\;\frak G_\bp=5$ \ then $M$ is locally holomorphically equivalent to the Heisenberg sphere $\mathbb H$.
  \item[$\bullet$] ${\rm dim}\;\frak G_\bp=1$ \  then $M$ is a circular hypersurface singularly umbilic at $\bp$.
  \item[$\bullet$] ${\rm dim}\;\frak G_\bp=0$ \ then $M$ is non-umbilic or semi-circular or generic singularly umbilic at $\bp$.
\end{itemize}
\end{Theorem}

Similarly, using the argument outlined in the introduction, we deduce the following result on the local symmetry group of a nondegenerate hypersurface.

\begin{Theorem}
Let $M$ be a nondegenerate hypersurface.  Then the local symmetry group of $M$ near a point $\bp$ has dimension $8$ if and only if $M$ is locally holomorphically equivalent to the Heisenberg sphere $\mathbb H$.  Otherwise, the dimension of the local symmetry group is $0,1,2,3$ depending upon whether the number of functionally independent differential invariants on a neighborhood of $\bp$ is $3,2,1,0$, respectively. In the maximally symmetric nonsingular case having a three-dimensional local symmetry group, $M$ is locally equivalent to an orbit of a three-dimensional subgroup of the holomorphic pseudo-group.
\end{Theorem}

\proof
If $M$ is totally singular, then it must be umbilic and hence locally congruent to the Heisenberg sphere.  Otherwise, the dimension of the symmetry group cannot be $\geq 4$, since, given that the dimension of $M$ is $3$, this would imply that every point $\bp \in M$ would have an isotropy group of dimension $\geq 1$, and hence be umbilic, and we are back to the previous case.  The connection between the dimension of the symmetry group and the number of functionally independent differential invariants at a nonsingular point follows from \cite[Theorem 5.17]{Olver-Fels-99}; see also \cite[Theorem 14.26]{Olver-1995}.
\endproof

The eight-dimensional case corresponds to the Heisenberg (and ordinary) sphere.  Cartan, \cite{Cartan-1932}, classified the homogeneous hypersurfaces possessing a three-dimensional symmetry group.  According to \cite[p.\ 656]{Beloshapka-Kossovskiy}, there are six classes; the \emph{tubular hypersurfaces}
\begin{itemize}
  \item[(t1) ]  $v = y^\lambda $ \ for \ $y > 0$, \ $\abs \lambda \geq 1$, \ $\lambda \ne 1,2$;
  \item[(t2) ]  $v = y \log y $ \ for \ $y > 0$;
  \item[(t3) ]  $r = e ^{a\, \varphi }$ \ for \ $a \geq 0$, \ where \ $y + \i v = r\, e ^{\i \varphi }$;
\end{itemize}
and the \emph{projective hypersurfaces}
\begin{itemize}
  \item[(p1) ]  $\hphantom-1 + \abs z^2 + \abs w^2 = a \, \abs{1 + z^2 + w^2}\hphantom-$ \ for \ $a > 1$;
  \item[(p2) ]  $\hphantom-1 + \abs z^2 - \abs w^2 = a \, \abs{1 + z^2 -w^2}\hphantom-$ \ for \ $a > 1$;
  \item[(p3) ]  $-1 + \abs z^2 + \abs w^2 = a \, \abs{-1 + z^2 + w^2}$ \ for \ $0 < \abs a < 1$.
\end{itemize}
We note that the tubular hypersurface of Type (t3) can be rewritten in the alternative implicit form
$$(y + \i v)^{1 - \i a}(y - \i v)^{1 + \i a} = 1. $$
The associated symmetry groups are as follows.  The tubular hypersurfaces all have nilpotent symmetry groups.  For the first, the symmetry algebra is spanned by
$${\rm(t1)} \hskip 40pt \partial_z, \qquad\partial_w,\qquad z\,\partial_z+\lambda \, w\,\partial_w.\hskip 100pt$$
The second tubular hypersurface has symmetry algebra
$${\rm(t2)} \hskip 40pt  \partial_z, \qquad\partial_w,\qquad z\,\partial_z+(z+ w)\,\partial_w.\hskip 80pt$$
As for the third, we have
$$ {\rm(t3)} \hskip 40pt \partial_z, \qquad\partial_w,\qquad (a\,z + w)\,\partial_z+(-z+ a\,w)\,\partial_w.\hskip 28pt $$
The symmetry algebra of the first class of  projective hypersurfaces is spanned by
$${\rm(p1)} \hskip 22pt -w\,\partial_z+z\,\partial_w,\qquad (1+z^2)\,\partial_z+z\,w \, \partial_w,\qquad z\,w\,\partial_z+(1+w^2)\,\partial_w.\hskip 56pt $$
We note that these vector fields span the complexification of the infinitesimal generators of the projective action of the rotation group  ${\rm SO}(3)$ on the projective plane $\mathbb{RP}^2$.  The other two projective hypersurfaces have similar symmetry generators
$${{\rm(p2)} \hskip 35pt  w\,\partial_z+z\,\partial_w,\qquad \ \  (1-z^2)\,\partial_z-z\,w \, \partial_w,\qquad \ -\,z\,w\,\partial_z+(1-w^2)\,\partial_w,\hskip4pt\atop
{\rm(p3)} \hskip 22pt -w\,\partial_z+z\,\partial_w,\qquad (-1+z^2)\,\partial_z+z\,w \, \partial_w,\qquad \ \ z\,w\,\partial_z+(-1+w^2)\,\partial_w.}\hskip 30pt $$
Because the tubular and projective hypersurfaces admit three-dimensional symmetry groups, the differential invariants $\CMJ,\CMK,\CML$ (and all higher order differential invariants) are all constant on them.

Turning to two-dimensional symmetry groups, away from singular points there are, up to holomorphic equivalence,  two locally transitive possibilities,  \cite{Olver-1995}.  The first is the abelian translation group whose Lie algebra is spanned by  $\partial _z, \partial _w$.  A translationally invariant hypersurface has the form
\begin{equation}
\label{vfy}
v = f(y) .
\end{equation}
We note that the Heisenberg sphere $v = z \oz$ is locally of the form \eqref{vfy}; replacing $z \mapsto e^{\i z}$, it becomes $v = e^{\i(z - \oz)}= e^{-2y}$.
We claim that, for a suitably generic function $f$, the dimension of the symmetry group is exactly $2$.  To see this, excluding the Heisenberg sphere, we need to show that a generic $f$ cannot admit a three-dimensional symmetry algebra.
A short computation shows that any Lie algebra containing the abelian translation algebra has a third generator of one of the two forms
$$e^{\alpha \, z + \beta \, w} (\gamma \, \partial _ z + \delta \, \partial _ w), \qquad (\alpha \, z + \beta \, w) \, \partial _ z + (\gamma \, z + \delta \,w)\, \partial _ w,$$
for $\alpha, \beta, \gamma, \delta \in \C$.

In the first case, we have $(\alpha , \beta), (\gamma, \delta ) \ne (0,0)$; moreover, since the symmetry algebra of the hypersurface must be real, $\alpha, \beta \in \R$.
The hypersurface \eqref{vfy} admits this generator if and only if $f(y)$ satisfies
\[
f'(y) = -\,\frac{{\rm Im} \> \delta \,e^{\i(\alpha \, y + \beta \, f(y))}}{{\rm Im} \> \gamma \,e^{\i(\alpha \, y + \beta \, f(y))}},
\]
which imposes severe restrictions on the form of $f(y)$.  In the second case, we can apply an invertible linear transformation with a  real\footnote{The transformation must be real so as not to complexify the translation generators.} coefficient matrix to $(z,w)$ to place the coefficient matrix of the vector field  into real canonical form, \cite[\S 3.4.1]{Horn-Johnson}, and hence there are three subcases
$$a \, z \, \partial _ z + b\,w\, \partial _ w,\qquad a \, z \, \partial _ z + (z + a\,w)\, \partial _ w,\qquad (a\,z + w)\,\partial_z+(-z+ a\,w)\,\partial_w,$$
where, by the reality of the symmetry algebra, $a,b \in \R$. It is easily shown that the corresponding hypersurfaces are, respectively, equivalent under a linear transformation in $w$ to Cartan's three tubular hypersurfaces (in the first two cases, there is no invariant hypersurface of the form \eqref{vfy} when $a=0$).

Any locally transitive non-abelian two-dimensional transformation group is equivalent to the one whose Lie algebra is spanned by $\partial _z, e^z\partial _w$.  In this case, the hypersurface equation takes the form
\begin{equation}
\label{zzz}
v = - u \tan y + f(y).
\end{equation}
Using a similar argument as above, only very particular functions $f(y)$ will admit a third symmetry generator.  We conclude that, in all cases, a hypersurface corresponding to a sufficiently generic $f(y)$ in \eqref{zzz} has only a two-dimensional symmetry group.

As for one-dimensional symmetry groups, one can, away from singular points, map the infinitesimal generator to $\partial _u$, which corresponds to translations in $u$.  The corresponding hypersurfaces are of the form $v = f(z,\oz)$, where $f$ is any real analytic function that does not depend on $u$.  One can clearly choose $f$ so that the hypersurface is nondegenerate and non-umbilic.  If the hypersurface were to admit a two-dimensional symmetry group, one could apply a holomorphic transformation to map it into the form \eqref{vfy} or \eqref{zzz}, and hence a hypersurface defined by a generic $f(z,\oz)$ only admits a one-dimensional symmetry group.

\vglue 15pt
%
\subsection*{Acknowledgments}

The authors would like to thank Howard Jacobowitz for his helpful comments and discussions during the preparation of this paper. We also thank Valerii Beloshapka for introducing us to \cite{Beloshapka-80} and helpful discussions. We further thank Niky Kamran for reading a draft version and sending many helpful comments.  The research of the second author was supported in part by a grant from IPM, No.\ 1400510415.
\newpage

\vglue10pt

\vglue 10pt

\end{document}